\documentclass{siamart1116}

\usepackage{amsfonts}
\usepackage{tikz}
\usepackage{pgfplots}
\usepackage{verbatim}
\usepackage[nolist]{acronym}

\newcommand{\tvector}[1]{\mathbf{ #1 }} % Macro for easy replacement with alternate vector notation.
\newcommand{\ttensor}[1]{\mathcal{ #1 }} % Todo: Find something better than \mathcal.
\newtheorem{rem}[theorem]{Remark}
\newenvironment{remark}{\begin{rem}\begin{rm}}{\end{rm}\end{rem}}

\newcommand{\TheTitle}{On Convergence %of Sequences
to Essential~Singularities} 

\newcommand{\TheAuthors}{Nathaniel J.\ McClatchey}

\headers{\TheTitle}{\TheAuthors}

\title{{\TheTitle}\thanks{January 4, 2018%Submitted to the editors June 28, 2017.%\today
\funding{Supported by the National Science Foundation under Grant 1418787.}}}

\author{
  Nathaniel J.\ McClatchey\thanks{Ohio University, Athens, OH
    (\email{nm160111@ohio.edu})%, \url{http://www.imag.com/\string~ddoe/}).
    }
}

\newacro{als}[ALS]{Alternating Least Squares}

\begin{document}
\maketitle

\begin{abstract}
An iterative optimization method applied to a function \(f\) on \(\mathbb{R}^n\) will produce a sequence of arguments \(\{\tvector{x}_k\}_{k \in \mathbb{N}}\); this sequence is often constrained such that \(\{f(\tvector{x}_k)\}_{k \in \mathbb{N}}\) is monotonic.
As part of the analysis of an iterative method, one may ask under what conditions the sequence \(\{\tvector{x}_k\}_{k \in \mathbb{N}}\) converges.
In 2005, Absil et~al.\ employed the {\L}ojasiewicz gradient inequality in a proof of convergence; this requires that the objective function exist at a cluster point of the sequence.
Here we provide a convergence result that does not require \(f\) to be defined at the limit \(\lim_{k \to \infty} \tvector{x}_k\), should the limit exist.
We show that a variant of the {\L}ojasiewicz gradient inequality holds on sets adjacent to singularities of bounded multivariate rational functions.
We extend the results of Absil et~al.\ to prove that if \(\{\tvector{x}_k\}_{k \in \mathbb{N}} \subset \mathbb{R}^n\) has a cluster point \(\tvector{x}_*\), if \(f\) is a bounded multivariate rational function on \(\mathbb{R}^n\), and if a technical condition holds, then \(\tvector{x}_k \to \tvector{x}_*\) even if \(\tvector{x}_*\) is not in the domain of \(f\).
We demonstrate how this may be employed to analyze divergent sequences by mapping them to projective space, and consider the implications this has for the study of low-rank tensor approximations.
\end{abstract}

% REQUIRED

% REQUIRED

\section{Introduction}	\label{sec:introduction}
Historically, a variety of approaches have been employed to examine the convergence of sequences produced by iterative optimization methods.
%	Local convergence
Some examine local convergence -- whether small perturbations of a solution are corrected, and at what rate.
Such results are typically proved directly -- as in the case of the Newton-Raphson method -- or by comparison to a method with known behavior \cite{COHEN:1972, USCHMA:2012}.	%	Todo:	Add more examples?
%	Global convergence
Others attempt to show global convergence -- that a particular method will produce a convergent sequence.
%As with local convergence, global convergence can sometimes be shown directly \cite{ES-HA-KH:2015, SH-YA-XI:2015}, but such proofs 
%To ease discussion, consider a sequence \(\{\tvector{x}_k\}_{k \in \mathbb{N}}\), a differentiable function \(f\), and its gradient \(\nabla f\).	%	Todo:	Fix awkwardness.
Though global convergence can be proved directly in some cases \cite{ES-HA-KH:2015, SH-YA-XI:2015}, more general approaches exist.
%	History
Since 1971, the Wolfe conditions \cite{WOLFE:1971} have been employed to ensure that the gradient of the objective function -- evaluated at the points of a sequence -- tends to zero.	%	Cite the Zoutendijk condition.
More recently, seminal work by Absil et~al.\ \cite{AB-MA-AN:2005} employed the {\L}ojasiewicz gradient inequality to provide a stronger result.

For convenience, we state the gradient form of Stanis{\l}aw {\L}ojasiewicz's theorem \cite{LOJASI:1959, LOJASI:1993} below, incorporating an improvement from \cite{SCH-USC:2015}.
%{\color{red} Make attribution of this theorem clearer.}
\begin{theorem}[{\L}ojasiewicz gradient inequality]	\label{thm:loj_ineq}
%	\footnote{This is a special case of the {\L}ojasiewicz inequality \cite{LOJASI:1959, LOJASI:1993} with an improvement noted in \cite{SCH-USC:2015}.}
	Let \(f\) be a real-analytic function on a neighborhood of \(\tvector{x}_*\) in \(\mathbb{R}^n\). Then there are constants \(c > 0\) and \(\theta \in (0, 1/2]\) such that
	\begin{equation}	\label{eqn:loj_ineq}
		|f(\tvector{x}) - f(\tvector{x}_*)|^{1-\theta}	\le	c \|\nabla f(\tvector{x})\|
	\end{equation}
	for any \(\tvector{x}\) in some neighborhood of \(\tvector{x}_*\).
\end{theorem}
%The {\L}ojasiewicz gradient inequality is the statement that
%\[	|f(\tvector{x}) - f(\tvector{x}_*)|^{1 - \theta}	\le	c \|\nabla f(\tvector{x})\|	\,,	\]
%for some constants \(\theta \in (0, 1/2]\) and \(c > 0\).
%If \(f : \mathbb{R}^n \to \mathbb{R}\) is a real-analytic function, then any point \(\tvector{x}_* \in \mathbb{R}^n\) has an open neighborhood on which the inequality holds.
Though the constant \(\theta\) can be determined in some cases \cite{BRZOST:2015} and estimated in others \cite{GWOZDZ:1999, OLEKSI:2013}, it is in general a~priori unknown.

%The works of Absil, Uschmajew, et~al.\ have been employed to analyze the convergence of line search \cite{SCH-USC:2015}, Alternating Least-Squares \cite{USCHMA:2015}, and Maximum Block Improvement \cite{LI-US-ZH:2015} methods.
%Some authors have devised similar results for more general functions using the Kurdyka-{\L}ojasiewicz inequality \cite{LAGEMA:2007, AT-BO-SV:2011, XU-YIN:2013, FR-GA-PE:2014}.
The result of Absil et~al.\ \cite{AB-MA-AN:2005} states that a sequence \(\{\tvector{x}_{k}\}_{k \in \mathbb{N}} \subset \mathbb{R}^n\) converges if it has a cluster point and if there is some real-analytic function \(f : \mathbb{R}^n \to \mathbb{R}\) such that the ``descent conditions'' (\ref{eqn:loj_assumption_1}--\ref{eqn:loj_assumption_2}) hold. % is real-analytic, there~exists \( \sigma > 0 \) such~that for~all sufficiently large \( k \)
%{\color{red}maybe: The first descent condition is ... ; it ensures that ...}
The first descent condition, that
\begin{equation}
	f(\tvector{x}_k) - f(\tvector{x}_{k+1})	\,\ge\,	\sigma\, \|\nabla f(\tvector{x}_k)\| \, \|\tvector{x}_{k+1} - \tvector{x}_{k}\|
    \tag{A1}	\label{eqn:loj_assumption_1}
    %\tag{\ref{eqn:loj_assumption_1}}
\end{equation}
for some \( \sigma > 0 \) and all sufficiently large \( k \), ensures that the sequence of \(f\)-values decreases sufficiently quickly.
%{eqn:loj_assumption_2}
The second descent condition prevents cycles by requiring that
\begin{equation}
	f(\tvector{x}_k) = f(\tvector{x}_{k+1}) \quad\implies\quad \tvector{x}_k = \tvector{x}_{k+1}	\,.
    \tag{A2}	\label{eqn:loj_assumption_2}
    %\tag{\ref{eqn:loj_assumption_2}}
\end{equation}
%holds, and the sequence \(\{\tvector{x}_k\}_{k \in \mathbb{N}}\) admits a cluster point, then that cluster point is the limit of the sequence \cite{AB-MA-AN:2005}.
Moreover, Uschmajew, et~al.\ proved that if also there~exists some \( \kappa > 0 \) such that for all sufficiently large \( k \),	%	Todo:	Is Uschmajew the originator?
\begin{equation}
	\|\tvector{x}_{k+1} - \tvector{x}_{k}\| \ge \kappa \| \nabla f(\tvector{x}_k)\| \,,
    \tag{A3}	\label{eqn:loj_assumption_3}
    %\tag{\ref{eqn:loj_assumption_3}}
\end{equation}
then the rate of convergence of the sequence may be bounded \cite{SCH-USC:2015, USCHMA:2015}.

In recent years, the results of Absil, Uschmajew, et~al.\  \cite{AB-MA-AN:2005, SCH-USC:2015, USCHMA:2015} have been of interest in the field of tensor approximation.
Given a target tensor \(\ttensor{T}\), one may attempt to find the tensor of a fixed rank \(r\) which best approximates the target.
This problem is often described as minimization of the ``error'' \(\|\ttensor{T} - \tau(\tvector{x})\|\), where \(\tvector{x} \in \mathbb{R}^n\) is a tuple of parameters, and \(\tau\) is a multilinear map from parameters to tensors of rank \(r\).
If \(r = 1\), there exist parameters \(\tvector{x}\) that minimize error, and several methods, including \ac{als}, are known to produce convergent sequences \cite{USCHMA:2015, LI-US-ZH:2015}.
If \(r \ge 2\), then the approximation \( \tau(\tvector{x}) \) is a sum of two or more separable tensors \(\ttensor{X}^1 + \dots + \ttensor{X}^r\).
Because no orthogonality constraint is imposed on \(\ttensor{X}^1, \dots, \ttensor{X}^r\), there exist problems for which error cannot be minimized \cite{BEY-MOH:2005, V-N-V-M:2014}.
For these ill-posed problems, if \(\|\ttensor{T} - \tau(\tvector{x}_k)\|\) is to approach its infimum, the sequence of parameters \(\{\tvector{x}_k\}_{k \in \mathbb{N}}\) must diverge.
As yet, little intuition exists for the behavior of these sequences; in particular, it is not known whether the summands \(\ttensor{X}^1, \dots, \ttensor{X}^r\) maintain some steady configuration, changing little except in scale, or whether they cycle among different configurations.

To gain intuition, the divergent case may be converted to a projective space by rescaling; one may then ask whether the normalized sequence \(\{\tvector{x}_k / \|\tvector{x}_k\|\}_{k \in \mathbb{N}}\) converges.
When this re-scaling is performed, however, one must alter the objective function such that it is optimized by a unit vector, as in \(\tvector{x} \mapsto \min_{\lambda \in \mathbb{R}} f(\lambda \tvector{x})\).
%This normalized sequence should not be expected to converge to a minimizer of \(f\); one must either multiply the sequence pointwise by some sequence of constants \(\{\lambda_k\}_{k \in \mathbb{N}}\), or alter the objective function such that it is optimized by a unit vector, as is demonstrated by \(\tvector{x} \mapsto \min_{\lambda \ge 0} f(\lambda \tvector{x})\).
%If the latter approach is applied to the tensor approximation problem, fitting the target tensor is then optimization of a bounded multivariate rational function in projective space.
In the tensor approximation problem, this results in a bounded multivariate rational function.
Targets for which the best approximation problem is ill-posed correspond to singularities of this function; some of these singularities are known to be essential --- that is, neither removable nor unbounded.

\begin{figure}[ht]
	\centering
	\begin{minipage}{.45\textwidth}
	\resizebox{.97\linewidth}{!}{
		\begin{tikzpicture}
  \begin{axis}[
      hide axis,
      %axis lines*=left,
      %axis equal image,
      view={-30}{40},
      xlabel = $x$,
      ylabel = $y$
    ]
    
%    \draw	(-2,2,0)	--	(-2,-2,0)	--	(2,2,0);

    \addplot3 [
      surf,
      domain=0:360,	y domain=0:2,
      samples=65,	samples y=12,
      line join=round,
%      shader=faceted interp,
      shader=faceted,
      variable=\t,
      z buffer = sort
    ]
    ( {cos(t)*y / max(abs(cos(t)), abs(sin(t)))},{sin(t)*y / max(abs(cos(t)), abs(sin(t)))},{-sin(2 * t) / (2 + 2 * (y / max(abs(cos(t)), abs(sin(t))))^2)} );
   
   \addplot3+[scatter,point meta = -0.5,color=black] table {figure_rational_pinch.dat};
   
   %\draw	(rel axis cs:0,0,0.5)	--	(rel axis cs:1,0,0.5)	--	(rel axis cs:1,1,0.5)	--	(rel axis cs:0,1,0.5)	--	cycle;
   \draw	(rel axis cs:0,1,0.5)	--	(rel axis cs:0,0,0.5)	--	(rel axis cs:1,0,0.5);	%--	(rel axis cs:1,1,0.5);
  \end{axis}
\end{tikzpicture}}
	\end{minipage}
	\begin{minipage}{.45\textwidth}
	\resizebox{.97\linewidth}{!}{
		\begin{tikzpicture}
  \begin{axis}[
      hide axis,
      %axis lines*=left,
      %axis equal image,
      view={20}{40},
      xlabel = $x$,
      ylabel = $y$
    ]
    
%    \draw	(-2,2,0)	--	(-2,-2,0)	--	(2,2,0);

    \addplot3 [
      surf,
      domain=0:360,	y domain=0:2,
      samples=65,	samples y=12,
      line join=round,
%      shader=faceted interp,
      shader=faceted,
      variable=\t,
      z buffer = sort
    ]
    ( {cos(t)*y / max(abs(cos(t)), abs(sin(t)))},{sin(t)*y / max(abs(cos(t)), abs(sin(t)))},{-sin(2 * t) / (2 + 2 * (y / max(abs(cos(t)), abs(sin(t))))^2)} );
   
   \addplot3+[scatter,point meta = -0.5,color=black] table {figure_rational_pinch.dat};
   
   %\draw	(rel axis cs:0,0,0.5)	--	(rel axis cs:1,0,0.5)	--	(rel axis cs:1,1,0.5)	--	(rel axis cs:0,1,0.5)	--	cycle;
   \draw	(rel axis cs:0,0,0.5)	--	(rel axis cs:1,0,0.5)	--	(rel axis cs:1,1,0.5);
  \end{axis}
\end{tikzpicture}}
	\end{minipage}
	\caption{Line Search along the gradient maximizes \(	f(x,y)	=	\frac{-xy}{(x^2 + y^2)(1 + x^2 + y^2)}	\) from an initial estimate of \( (x_0, y_0) = (2, -0.1) \).}
	\label{fig:rational_gradient_descent}
\end{figure}
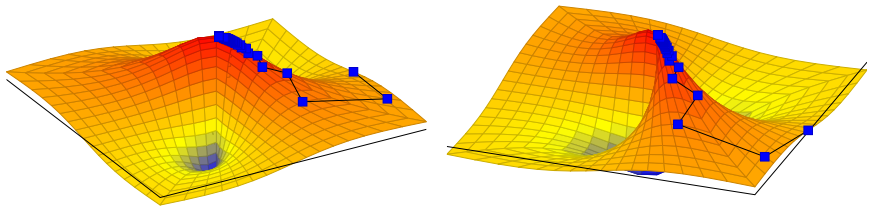
Wherever a function fails to be continuous, the {\L}ojasiewicz gradient inequality cannot apply.
This does not prevent optimization methods from producing convergent sequences, as is illustrated in \cref{fig:rational_gradient_descent}, but it does preclude the use of the theorems of Absil, Uschmajew, et~al.
Fortunately, those theorems may be strengthened to require only that the {\L}ojasiewicz gradient inequality hold on points in the sequence.

The conditions required by the following theorems will follow from the main results of this paper.
Specifically, \cref{thm:Lojasiewicz_cones,thm:Lojasiewicz_on_sequence} will establish \eqref{eqn:loj_on_sequence} for sequences such as that illustrated in \cref{fig:rational_gradient_descent}.
\begin{theorem}	\label{thm:loj_converge_modified}
	Let \(U \subset \mathbb{R}^n\) be an open set, let \( f : U \to \mathbb{R} \) be a differentiable function, and let \( \{ \tvector{x}_k\}_{k=1}^{\infty} \subset U \) be a sequence of vectors satisfying Assumptions~\eqref{eqn:loj_assumption_1} and~\eqref{eqn:loj_assumption_2}.
	If a cluster point $\tvector{x}^*$ of the sequence $\{ \tvector{x}_k\}_{k=1}^{\infty}$ admits an open neighborhood ${V \subset \mathbb{R}^n}$ and constants ${\theta \in (0, 1/2]}$ and ${c \in \mathbb{R}}$ such that for~all~$k$,
	\begin{equation}	\label{eqn:loj_on_sequence}%	\label{EQN:LOJ_ON_SEQUENCE}
		\tvector{x}_k \in V	\quad \implies \quad	| f(\tvector{x}_k) - \lim_{j \to \infty} f(\tvector{x}_j) |^{1-\theta} \le c \|\nabla f(\tvector{x}_k)\|	\tag{\texorpdfstring{{\text{\L}}\textsubscript{*}}{L*}}
	\end{equation}
	then \( \tvector{x}^* \) must be the limit of the sequence  \( \{ \tvector{x}_k\}_{k=1}^{\infty} \).
\end{theorem}

\begin{theorem}	\label{thm:loj_linear_modified}
	Under the conditions of \cref{thm:loj_converge_modified}, if also Assumption~\eqref{eqn:loj_assumption_3} holds, then ${\nabla f(\tvector{x}_k) \to \tvector{0}}$ and
	\begin{equation*}
		\|\tvector{x}^* - \tvector{x}_k\|	=
			\left\{\begin{matrix}
				O(q^k) & \quad \text{ if } & \theta = \frac12 &\text{ (for some } 0 < q < 1 \text{),}\\
				O(k^{\frac{-\theta}{1 - 2\theta}}) & \quad \text{ if } & 0 < \theta < \frac12&
			\end{matrix}\right.
	\end{equation*}
	where \( \theta \) is such that \eqref{eqn:loj_on_sequence} holds.
\end{theorem}
The following short proof lists only those changes required to extend Uschmajew's result.
For a full proof of \cref{thm:loj_converge_modified,thm:loj_linear_modified}, see~\cref{apx:full_proof_lojasiewicz}.
\begin{proof}[Proof of \Cref{thm:loj_converge_modified,thm:loj_linear_modified}]
Though \(f(\tvector{x}^*)\) is not assumed to exist, \eqref{eqn:loj_on_sequence} implies that \( \lim_{k \to \infty} f(\tvector{x}_k) \) exists and is finite.
With this established, only three changes are needed to extend the proof in \cite[p.~644]{SCH-USC:2015} to a proof of \Cref{thm:loj_converge_modified,thm:loj_linear_modified}:
\begin{enumerate}
\item Replacing all occurences of \( f(\tvector{x}^*) \) with \( \lim_{k \to \infty} f(\tvector{x}_k) \) we remove the requirement that \( \tvector{x}^* \) be in the domain of \( f \).
\item Inequality (A.1) in \cite{SCH-USC:2015} need only hold when $\tvector{x} = \tvector{x}_k$. Thus, it follows from \eqref{eqn:loj_on_sequence}.
\item The statement that \(n\) may be selected so large that \( \| \tvector{x}_n - \tvector{x}^* \| < \frac{\epsilon}{2} \) and \( \frac{\Lambda}{\sigma \theta} f_n^{\theta} < \frac{\epsilon}{2} \) does not require that \( f \) be continuous at \( \tvector{x}^* \).
Instead, it follows from the existence of~\( \lim_{k \to \infty} f(\tvector{x}_k) \) and the assumption that \( \tvector{x}^* \) is a cluster point.
\end{enumerate}
\end{proof}

\begin{remark}
\Cref{thm:loj_converge_modified,thm:loj_linear_modified} are strictly stronger than the results of Absil, Uschmajew, et~al.\ in that weaker hypotheses allow the same conclusions.
Specifically, if \( f \) is analytic on a neighborhood of \(\tvector{x}^*\), \eqref{eqn:loj_on_sequence} follows from \cref{thm:loj_ineq}.
\end{remark}

Our \cref{thm:loj_converge_modified,thm:loj_linear_modified} do not require continuity of the objective function, but do require a weaker form \eqref{eqn:loj_on_sequence} of the {\L}ojasiewicz gradient inequality.
Beyond \cref{thm:loj_converge_modified,thm:loj_linear_modified}, the contributions of this paper are four-fold.
\begin{itemize}
	\item	In \cref{sec:loj_coneineq}, we contribute a version of the {\L}ojasiewicz gradient inequality that may be employed at essential singularities of bounded multivariate rational functions.
		This inequality holds not on a neighborhood of a singularity, but instead on open sets which have boundaries containing the singularity.
	\item	\Cref{sec:loj_seq_tail} establishes that our {\L}ojasiewicz gradient inequality holds on the tail of the sequence, under a technical condition.
		Specifically, if the limiting behavior of \(f\) near the singularity does not depend continuously on the direction of approach, then the sequence must avoid the ``unsafe'' directions at which discontinuities occur.
		We show that the set of ``unsafe'' directions of approach is closed and has Lebesgue measure zero.
	\item	\Cref{sec:convergence_result} combines these results to prove that if this and (\ref{eqn:loj_assumption_1}--\ref{eqn:loj_assumption_3}) hold, then the sequence \(\{\tvector{x}_{k+1}\}_{k \in \mathbb{N}}\) converges.
	\item	Finally, \cref{sec:application} provides additional tools for examination of convergence in direction and briefly discusses the implications for tensor approximation.
\end{itemize}
This partially closes a gap in the theory of tensor approximation, and provides a general tool for analysis of sequences.

%%%%%%%%%%%%%%%%%%%%%%%%%%%%%%%%%%%%%%%%%%%%%%%%%%%%%%%%%%%%%%%%

\section[On the assumption of a {\L}ojasiewicz inequality]{On the assumption (\texorpdfstring{\lowercase{\ref{eqn:loj_on_sequence}}}{L*})}	\label{sec:loj_sequence}
Before we begin our analysis of \eqref{eqn:loj_on_sequence}, we note two obstacles and the means by which we circumvent them.

First, we must assume that \(f\) has some structure strict enough that its behavior may be analyzed near a singularity \(\tvector{x}^*\), yet not so strict as to require that \(f\) be continuous at \(\tvector{x}^*\).
Throughout the paper, we will consider multivariate rational functions on~\(\mathbb{R}^n\), with the assumption that the domain of such functions is defined implicitly to be all points in~\(\mathbb{R}^n\) at which division by zero does not occur.
More formally, we define a multivariate rational function as follows:
\begin{definition}
A function \(r\) is a multivariate rational function on \(\mathbb{R}^n\) if and only if there exist multivariate polynomials \(p, q : \mathbb{R}^n \to \mathbb{R}\) such~that
\(r = \frac{p}{q}\).
\end{definition}
%Before we can study the behavior of a function \(f\) near its singularities, we must assume that \(f\) has some structure.

Second, we cannot expect the {\L}ojasiewicz gradient inequality to hold on open neighborhoods of \(\tvector{x}^*\).
As part our analysis of \eqref{eqn:loj_on_sequence}, we find sets on which
\begin{equation}	\label{eqn:Lojasiewicz_generalized}
	|f(\tvector{x}) - K|^{1-\theta} \le c\|\nabla f(\tvector{x})\|
\end{equation}
holds.
If \(\tvector{x}^*\) were a removable singularity of \(f\), one might be able to establish \eqref{eqn:Lojasiewicz_generalized} on an open neighborhood of \(\tvector{x}^*\).
For an essential singularity, however, \eqref{eqn:Lojasiewicz_generalized} can fail to hold on open neighborhoods of \(\tvector{x}^*\).

To illustrate this, consider the function defined by \((x,y) \overset{f}{\longmapsto} \frac{y^2 - x^2}{x^2 + y^2}\) and the sequence \(\tvector{e}_1, \frac{1}{2} \tvector{e}_2, \frac{1}{4} \tvector{e}_1, \frac{1}{8} \tvector{e}_2, \dots \subset \mathbb{R}^2\).
Clearly, \(\tvector{x}_k \to \tvector{0}\).
Moreover, for~all \(k \in \mathbb{N}\), we have \(\nabla f(\tvector{x}_k) = \tvector{0}\) and \(f(\tvector{x}_k) = (-1)^k\).
If \eqref{eqn:Lojasiewicz_generalized} were to hold on any open neighborhood of \(\tvector{0}\), it would imply that \(f(\tvector{x}_k) = K\) for~all sufficiently large \(k\), which contradicts our calculation.

Rather than examine open neighborhoods of singularities, we establish that \eqref{eqn:Lojasiewicz_generalized} holds on certain open sets adjacent to singularities of bounded multivariate rational functions.
We then provide conditions under which sequences \(\{\tvector{x}_k\}_{k \in \mathbb{N}}\) remain within those open sets.

\subsection{A {\L}ojasiewicz-like inequality holds on cones}	\label{sec:loj_coneineq}
Bounded multivariate rational functions may admit essential singularities, such as that illustrated in \cref{fig:rational_gradient_descent}.
Bounded univariate rational functions instead exhibit only removable singularities.
%We expl
In this \namecref{sec:loj_coneineq}, we parameterize multivariate functions in terms of directions and distance; that is, as images of lines under a multivariate function.
\Cref{thm:rational_taylor,thm:convergent_taylor} establish that these parameterizations are themselves analytic, if not rational.
The parameterization is employed in \Cref{thm:Lojasiewicz_cones} to establish a {\L}ojasiewicz inequality on cones near singularities of the original function.

%%%%%%%%%%%%%%%%%%%%%%%%%%%%%%%%%%%%%%%%%%%%%%%%%%%%%%%%%%%%%%%%%%%%%%%%%%%%
%                Taylor coefficients are analytic.                         %
%%%%%%%%%%%%%%%%%%%%%%%%%%%%%%%%%%%%%%%%%%%%%%%%%%%%%%%%%%%%%%%%%%%%%%%%%%%%
\begin{lemma}	\label{thm:rational_taylor}
Suppose $f$ is a multivariate rational function on \({\mathbb{R}^m \times \mathbb{R}}\).
If
	\[	f_{\tvector{x}}(t) = \lim_{s \to t} f(\tvector{x}, s)	\]
is defined for some ${\tvector{x} \in \mathbb{R}^m}$ and all ${t \in \mathbb{R}}$, then there exists an open set ${\mathcal{O} \subset \mathbb{R}^m}$ of full measure and multivariate rational functions $c_n : \mathcal{O} \to \mathbb{R}$ defined everywhere on $\mathcal{O}$ such that, for every $\tvector{x} \in \mathcal{O}$ there exists $\rho_{\tvector{x}} > 0$ such that
	\[	|t| < \rho_{\tvector{x}}	\quad\implies\quad	f_{\tvector{x}}(t) = \sum_{n=0}^{\infty} c_n(\tvector{x}) t^n	\,.	\]
\end{lemma}
%    Proof
\begin{proof}
Suppose $f_{\tvector{x}}(t) = \lim_{s \to t} f(\tvector{x}, s)$ is defined for all $t \in \mathbb{R}$.
By fixing $\tvector{x} \in \mathbb{R}^m$, we reduce $f(\tvector{x}, t)$ to a univariate rational function with respect to $t$, so its continuous extension $f_{\tvector{x}}(t)$ is a univariate rational function. Because $f_{\tvector{x}}(t)$ is defined everywhere, it is everywhere real-analytic.

%We derive useful properties of the Taylor expansion, around $t = 0$, of this continuous extension.
Let
	\[	f(\tvector{x}, t) = \frac{f_{\mathrm{numer}}(\tvector{x}, t)}{f_{\mathrm{denom}}(\tvector{x}, t)}	\,,	\]
where $f_{\mathrm{numer}}$ and $f_{\mathrm{denom}}$ are multivariate polynomials. $f_{\tvector{x}}(t)$ is analytic, so the coefficients $c_n$ of its Taylor expansion around $0$ exist, and for $n = 0, \dots, \infty$,
	\[	c_n(\tvector{x}) = \frac{1}{n!} \lim_{t \to 0} \frac{d^n}{dt^n} \frac{f_{\mathrm{numer}}(\tvector{x}, t)}{f_{\mathrm{denom}}(\tvector{x}, t)}	\,.	\]
Because the function \(f_{\tvector{x}}\) is analytic for each \(\tvector{x}\), there exist \(\rho_{\tvector{x}}\) such that \(|t| < \rho_{\tvector{x}}\) implies \( f_{\tvector{x}}(t) = \sum_{n=0}^{\infty} c_n(\tvector{x}) t^n\).
Note that the neighborhood of convergence depends on the direction parameter \(\tvector{x}\).
We now turn our attention to properties of the coefficients \(c_n(\tvector{x})\).

By repeated application of the quotient rule, and omitting the numerators because they are not required for the proof, we obtain an expression of the form
	\[	c_n(\tvector{x}) = \frac{1}{n!} \lim_{t \to 0} \frac{\dots}{(f_{\mathrm{denom}}(\tvector{x}, t))^{2^n}}	\]
which is a limit of a rational function. These limits exist by analyticity of $f_{\tvector{x}}(t)$, so they may be evaluated by repeated application of L'Hospital's rule.

Let $f_n(\tvector{x})$ denote the coefficients of the Taylor expansion, with respect to~$t$, of $f_{\mathrm{denom}}(\tvector{x}, t)$. Note that these are multivariate polynomials with respect to~$\tvector{x}$. Let
	\[	n_{\mathrm{min}}	=	\min \{n \in \mathbb{Z} : n \ge 0 \text{ and } \exists \, \tvector{x} \text{ st. }f_n(\tvector{x}) \ne 0\}	\,,	\]
and define
\begin{equation}	\label{eqn:set_lines_analytic}
	\mathcal{O}	=	\{ \tvector{x} : f_{n_{\mathrm{min}}}(\tvector{x}) \ne 0\}	\,.
\end{equation}
It must be shown that \(n_{\mathrm{min}}\) is well-defined.
The function $f_{\tvector{x}}$ is assumed to be defined for some $\tvector{x} \in \mathbb{R}^m$, so $f_{\mathrm{denom}}$ is nonzero at some \((\tvector{x},t)\).
From this and analyticity of $f_{\mathrm{denom}}$, it follows that $f_n(\tvector{x})$ is non-zero for some \(n\).
This establishes that \(n_{\mathrm{min}}\) is well-defined.

A multivariate real or complex polynomial is either identically zero or non-zero almost-everywhere with respect to Lebesgue measure.
Further, the set on which such a polynomial is equal to a given constant is closed.
Thus $\mathcal{O}$ is an open set of full measure.

For every $\tvector{x} \in \mathcal{O}$, the coefficients $c_n(\tvector{x})$ may be evaluated by applying L'Hospital's rule exactly $2^n n_{\mathrm{min}}$ times. This implies that on $\mathcal{O}$, every $c_n$ is a multivariate rational function with respect to $\tvector{x}$, and is defined everywhere on $\mathcal{O}$.
\end{proof}

% {thm:analytic_taylor}
%\begin{corollary}	\label{thm:analytic_taylor}
%Because rational functions are analytic on their domains, the coefficients $c_n$ in \cref{thm:rational_taylor} are real-analytic on $\mathcal{O}$.
%\end{corollary}
%\begin{proof}
%Rational functions are analytic on their domains.
%\end{proof}

%%%%%%%%%%%%%%%%%%%%%%%%%%%%%%%%%%%%%%%%%%%%%%%%%%%%%%%%%%%%%%%%%%%%%%%%%%%%
%                Radius of convergence is continuous on line direction.    %
%%%%%%%%%%%%%%%%%%%%%%%%%%%%%%%%%%%%%%%%%%%%%%%%%%%%%%%%%%%%%%%%%%%%%%%%%%%%

It is well known that the trailing coefficients of the Taylor series expansion of a univariate rational function satisfy a linear recurrence relation.
For completeness, we provide a proof of this here:
\begin{proposition}	\label{thm:rational_taylor_recurrence}
Let $\mathbb{K} = \mathbb{R}$ or $\mathbb{C}$, let $f, g : \mathbb{K} \to \mathbb{K}$ be polynomials of degree $d_f$ and $d_g$ given as
\( f(x)	=	\sum_{i=0}^{d_f} f_i x^i \) and \( g(x)	=	\sum_{j=0}^{d_g} g_j x^j \),
and let $\{c_n\}_{n=0}^{\infty} \subset \mathbb{K}$ such that $\left(\frac{f}{g}\right)(x) = \sum_{n=0}^{\infty} c_n x^n$ on some open $U \subset \mathbb{K}$. Then for~all $n > \max \{d_f, d_g\}$, the coefficients \(c_n\) are given by \( c_n	=	-\sum_{j=1}^{d_g} c_{n-j} \frac{g_j}{g_0}	\).
\end{proposition}
\begin{proof}
For all $x \in U$, we have \(\frac{f(x)}{g(x)} = \sum_{n=0}^{\infty} c_n x^n\).
Then, multiplying by \(g(x)\),
\begin{align*}
	f(x)
	&=	g(x)\textstyle\sum_{n=0}^{\infty} c_n x^n & \implies
\\
	\textstyle\sum_{i=0}^{d_f} f_i x^i
		&=	\left(\textstyle\sum_{j=0}^{d_g} g_j x^j \right) \left(\textstyle\sum_{n=0}^{\infty} c_n x^n \right) &
\\
		&=	\textstyle\sum_{n=0}^{\infty} \left( x^n \textstyle\sum_{j=0}^{\min \{ n, d_g \}} c_{n-j} g_j\right)\,. &
\end{align*}
Matching terms gives the equation \( 0 = \sum_{j=0}^{d_g} c_{n-j} g_j \) for~all \(n > \max \{ d_f, d_g \}\), and solving for $c_n$ completes the proof.
\end{proof}

We have shown that the Taylor series expansion in \cref{thm:rational_taylor} converges within some radius of convergence $\rho_{\tvector{x}}$, but this radius of convergence might depend on $\tvector{x}$.
Before we can use \cref{thm:rational_taylor} for its intended purpose, we must show that $\rho_{\tvector{x}}$ may be bounded away from 0.
%every point $\tvector{x}$ admits a neighborhood in which $\rho_{\tvector{x}}$ may be assumed to be constant.
We will do this by showing that $\rho_{\tvector{x}}$ depends continuously on $\tvector{x}$ so that we may later apply a compactness argument.

\begin{lemma}	\label{thm:convergent_taylor}
If the maps $c_n$ are defined as in \cref{thm:rational_taylor}, then there exists a continuous map $\rho : \mathcal{O} \to (0, \infty)$ and a corresponding set
	\[	U_{\rho} = \{\, (\tvector{x}, t) \,|\, \tvector{x} \in \mathcal{O}, \, t \in (-\rho(\tvector{x}), \rho(\tvector{x})) \,\}	\]
on which $\sum_{n=0}^{\infty} \left|c_n(\tvector{x}) t^n\right|$ converges uniformly.
\end{lemma}

\begin{proof}
By \cref{thm:rational_taylor_recurrence}, there exist $N, r \in \mathbb{N}$ and matrices $C_{\tvector{x}} \in \mathbb{R}^{r \times r}$ such that for all $n > N$,
	\[	\begin{bmatrix}c_{n+1}(\tvector{x}) \\ \vdots \\ c_{n+r+1}(\tvector{x}) \end{bmatrix} = C_{\tvector{x}} \begin{bmatrix}c_n(\tvector{x}) \\ \vdots \\ c_{n+r}(\tvector{x}) \end{bmatrix}	\,.	\]
Further, the recurrence matrices $C_{\tvector{x}}$ and its $\infty$-norm $\|C_{\tvector{x}}\|_{\infty}$ depend continuously on the choice of $\tvector{x} \in \mathcal{O}$.
We may thus bound the coefficients $c_n$ by a geometric sequence and control the convergence of the terms $c_n t^n$ by choice of $t$.
Select $N \in \mathbb{N}$ such that recurrence relation holds for all $n \ge N - r$. Note that $N$ may be chosen identically for all $\tvector{x} \in \mathcal{O}$.
For any $0 < k < 1$, selecting $t$ such that
\begin{equation} \label{eqn:convergent_taylor_bound}
|t| \le \frac{k}{\max \{ 1, \|C_{\tvector{x}}\|_{\infty}\}}
\end{equation}
guarantees that for $n \ge N$, we may bound $|c_n t^n|$ by
\[	\begin{aligned}
	|c_n(\tvector{x}) t^n|
    	&\le	\left( \max_{j=0,\dots,r} |c_{N - j}(\tvector{x})| \right) \, \|C_{\tvector{x}}\|_{\infty}^{n-N} \left(\frac{k}{\max \{ 1, \|C\|_{\infty}\}}\right)^n
\\
		&\le	\left( \max_{j=0,\dots,r} |c_{N - j}(\tvector{x})| \right)  k^{n}
\,. \end{aligned}	\]
Then
\[ \sum_{n = N}^{\infty} |c_n(\tvector{x}) t^n| \le k^N \left(\frac{1}{1 - k}\right) \max_{j=0,\dots,r} |c_{N - j}(\tvector{x})| \,. \]
This may be bounded arbitrarily by choice of $k$. In particular, for a given bound $\epsilon > 0$, the above bound shows that $\sum_{n = N}^{\infty} |c_n t^n| \le \epsilon$ if
	\[	\left( \max_{j=0,\dots,r} |c_{N - j}(\tvector{x})| \right) k^N + \epsilon k - \epsilon = 0	\,.	\]
This equation has a solution $k = k_{\tvector{x},\epsilon} \in (0, 1)$ which depends continuously on $\tvector{x}$~and~$\epsilon$.
Fix~$\epsilon > 0$, and define $\rho$ by
	\[	\tvector{x} \quad \overset{\rho}{\longmapsto} \quad \frac{k_{\tvector{x},\epsilon}}{2 \max \{ 1, \|C_{\tvector{x}}\|_{\infty}\}}	\,.	\]
Note that this map is continuous with respect to $\tvector{x}$, and that
	\[	|t| \le 2\rho(\tvector{x})	\quad\implies\quad	\sum_{n = N}^{\infty} |c_n(\tvector{x}) t^n| \le \epsilon	\,.	\]
Then, for $|t| \le \rho(\tvector{x})$ and $n \ge N$, the terms $|c_n(\tvector{x}) t^n|$ are bounded by
	\[	|c_n(\tvector{x}) t^n| \le \frac{\epsilon}{2^n}	\,.	\]
The series $\sum_{n = N}^{\infty} \frac{\epsilon}{2^n}$ converges, so $\sum_{n = N}^{\infty} |c_n(\tvector{x}) t^n|$ converges uniformly for $(\tvector{x}, t) \in U_{\rho}$.
\end{proof}

%%%%%%%%%%%%%%%%%%%%%%%%%%%%%%%%%%%%%%%%%%%%%%%%%%%%%%%%%%%%%%%%%%%%%%%%%%%%
%                         Lojasiewicz-like inequality                      %
%%%%%%%%%%%%%%%%%%%%%%%%%%%%%%%%%%%%%%%%%%%%%%%%%%%%%%%%%%%%%%%%%%%%%%%%%%%%

We now define \(\mathcal{O}_{\tvector{x}^*}\) the largest set of directions from which \(\tvector{x}^*\) can be approached while ensuring that the parameterized function \((\tvector{x},t) \mapsto f(t \tvector{x} + \tvector{x}^*)\) behaves consistently.
\begin{definition}	\label{def:safe_directions}
Given a bounded multivariate rational function \( f \) on \(\mathbb{R}^n\) and a point \( \tvector{x}^* \in \mathbb{R}^n \), and denoting by \(f_n(\tvector{x})\) the \(n\)-th Maclaurin series coefficient, with respect to \(t\), of the denominator of the rational function defined by \((\tvector{x}, t) \mapsto f(t \tvector{x} + \tvector{x}^*)\), we define
	\[	\mathcal{O}_{\tvector{x}^*}	\quad=\quad	\{\, \tvector{x} : f_{n_{\mathrm{min}}}(\tvector{x}) \ne 0 \,\}	\,,	\]
where \( n_{\mathrm{min}} = \min\{\, n \in \mathbb{Z}_{\ge0} \,|\, \exists \, \tvector{x} : f_{n}(\tvector{x}) \ne 0 \,\} \).
\end{definition}

\Cref{def:safe_directions} will be used to describe those directions of approach near which the function \(f\) is sufficiently well-behaved.
Later sections will demonstrate that if a sequence admits a subsequence that approaches \(\tvector{x}^*\) along a direction in \(\mathcal{O}_{\tvector{x}^*}\), then \(\tvector{x}^*\) is the limit of that sequence.
That is, approaching \(\tvector{x}^*\) from almost any direction will trap the sequence at that cluster point.

Note that \( \mathcal{O}_{\tvector{x}^*} \) satisfies the conditions of the open set specified by \cref{thm:rational_taylor} for the function \( f_{\tvector{x}}(t) = \lim_{s \to t} f(s \tvector{x} + \tvector{x}^*) \), as it is identical to the set constructed in the proof of \cref{thm:rational_taylor}.
Note also that for any \(t \in \mathcal{R} \setminus \{0\}\) and \(\tvector{x} \in \mathcal{O}_{\tvector{x}^*}\), we have \( t\tvector{x} \in \mathcal{O}_{\tvector{x}^*} \).

We are now prepared to show that a generalized {\L}ojasiewicz gradient inequality \eqref{eqn:Lojasiewicz_generalized} holds near the singularities of bounded multivariate rational functions.
\begin{theorem}	\label{thm:Lojasiewicz_cones}
Suppose $E$ is a bounded multivariate rational function on~$\mathbb{R}^m$.
For any point $\tvector{p} \in \mathbb{R}^m$ and any direction $\tvector{d} \in \mathcal{O}_{\tvector{p}} \setminus \{\tvector{0}\}$, where \(\mathcal{O}_{\tvector{p}}\) is as in \cref{def:safe_directions}, there exist an open set $U_{\bowtie} \subset \mathbb{R}^m$ and constants ${\theta \in (0, \frac12]}$ and $k > 0$ such that for all $\tvector{x} \in \operatorname{domain}(E) \cap U_{\bowtie}$,
%For any point $\tvector{p} \in \mathbb{R}^m$ and almost any direction $\tvector{d} \in \mathbb{R}^m \setminus \{\tvector{0}\}$, there exist an open set $U_{\bowtie} \subset \mathbb{R}^m$ and constants ${\theta \in (0, \frac12]}$ and $k > 0$ such that for all $\tvector{x} \in \operatorname{domain}(E) \cap U_{\bowtie}$,
\begin{equation}	\label{eqn:Lojasiewicz_inequality_cone}
	\left|E(\tvector{x}) - \lim_{s \to 0} E(\tvector{p} + s \tvector{d})\right|^{1 - \theta} \le k \left\|\nabla E(\tvector{x})\right\| \,.
\end{equation}
Further, $U_{\bowtie}$ may be chosen so that
\begin{enumerate}
	\item there exists some $s \in (0, 1]$ for which $\tvector{p} + s\tvector{d} \in U_{\bowtie}$, and
	\item for all $\tvector{y} \in U_{\bowtie}$ and $t \in [-1, 0) \cup (0, 1]$, the set $U_{\bowtie}$ contains $\tvector{p} + t(\tvector{y} - \tvector{p})$.
\end{enumerate}
\end{theorem}

\begin{proof}
If $E$ is nowhere defined, the theorem holds vacuously.
We assume that $E$ is somewhere defined.

%\subsection{Definition of the Line-Restricted Objective Function}
Consider lines in parameter space, beginning at $\tvector{p}$ and with direction $\tvector{d} \in \mathbb{R}^m$:
$$\tvector{p} + t \tvector{d}, \quad t \in \mathbb{R}\,.$$
$E(\tvector{p} + t \tvector{d})$ is a univariate bounded real rational function, with respect to~$t$. As such, it either exists nowhere or the limit $\lim_{s \to t} E(\tvector{p} + s \tvector{d})$ exists for all \(t \in \mathbb{R}\).
If the limit exists, it is a univariate real rational function defined everywhere on $\mathbb{R}$, and thus it is everywhere real-analytic.

%{\color{blue}Note that a line in parameter space is not, in general, a line in tensor space.}

%Let$$f(\tvector{x}, t) = E(\tvector{p} + t \tvector{x}).$$
By \cref{thm:rational_taylor}, there exists some open subset $\mathcal{O} \subset \mathcal{R}^m$ of full measure on which there are defined rational functions $c_n : \mathcal{O} \to \mathbb{R}$ such that for all $\tvector{y} \in \mathcal{O}$,
	\[	\lim_{s \to t} E(\tvector{p} + s \tvector{y})	=	\sum_{n=0}^{\infty} c_n(\tvector{y}) t^n	\]
if $|t| < \rho_{\tvector{y}}$.
This may be selected such that \(\mathcal{O} = \mathcal{O}_{\tvector{p}}\), as in \cref{def:safe_directions}.
By \cref{thm:convergent_taylor}, select $\rho_{\tvector{y}}$ to depend continuously on $\tvector{y}$.

%%%%%%%%%%%%%%%%%%%%%%%%%%%%%%%%%%%%%%%%%%%%%%%%%%%%%%%%%%%%%%%%%%%%%%%%%%%%
%                Convergence in single neighborhood                        %
%%%%%%%%%%%%%%%%%%%%%%%%%%%%%%%%%%%%%%%%%%%%%%%%%%%%%%%%%%%%%%%%%%%%%%%%%%%%

Select a reference direction $\tvector{d} \in \mathcal{O}$.
We may assume, by rescaling, that $\left\|\tvector{d}\right\| = 1$.
The parameter space %Todo: Symbol
is a finite-dimensional Banach space and $\mathcal{O}$ is open, so there exist a compact set $K$ and an open neighborhood $U$ of $\tvector{d}$ such that
\begin{equation*}	\label{eqn:nice_directions}
	\tvector{d} \in U \subset K \subset \mathcal{O}\,.
\end{equation*}
\(\rho_{\tvector{x}}\) is continuous on \(K\) and on \(U\).
Let \(M = \min_{\tvector{x} \in K} \rho_{\tvector{x}}\).
Because \(K\) is compact, \(M\) is well-defined.
Further, $\rho_{\tvector{x}} > 0$ by definition, so \(M > 0\).
Define
%its lower bound $M$ {\color{red} ?on $K$?} is positive. Define
\begin{equation}	\label{eqn:nice_neighborhood}
	U_{\tvector{d}} = \left\{ \begin{bmatrix}	\tvector{y}	\\	t	\end{bmatrix} : \tvector{y} \in U \quad\text{ and }\quad |t| < M \right\}\,.
\end{equation}
\\
%Ice Cream Cone proof
Define the extended function $F$ by
\begin{equation}	\label{eqn:extended_function}
	F(\tvector{y}, t)  = \lim_{s \to t} E(\tvector{p} + s \tvector{y}) = \sum_{n=0}^{\infty} c_n(\tvector{y}) t^n \,.
\end{equation}
Within $U_{\tvector{d}}$, \eqref{eqn:extended_function} is a convergent sum of analytic functions, and thus \eqref{eqn:extended_function} is itself a real-analytic function. By \cref{thm:loj_ineq} there exist an open ball \(V \subset U_{\tvector{d}}\) centered at \(\begin{bmatrix}	\tvector{d}	\\	0	\end{bmatrix}\) and constants \(\theta \in \left(0, \frac12\right]\) and \(k > 0\), such that for all \(\tvector{y} \in V\),
\begin{equation}	\label{eqn:linearized_Lojasiewicz}
	|F(\tvector{y}, t) - F(\tvector{d}, 0)|^{1 - \theta} \le \frac{k}{2} \|\nabla F(\tvector{y}, t)\| \,.
\end{equation}
%Without loss of generality, we assume that \(V\) is an open ball centered at \(\begin{bmatrix}	\tvector{d}	\\	0	\end{bmatrix}\).
%
Define the open set $V_* \subset V$ by
	\[	V_*	=	\left\{ \begin{bmatrix}\tvector{y} \\ t\end{bmatrix} \in V : \left\|\tvector{y} - \tvector{d}\right\| < \frac12 \quad \text{ and } \quad 0 < |t| < \frac{1}{2} \right\}	\,.	\]
From the definition of \(V\) as an open ball centered at \(\begin{bmatrix}	\tvector{d}	\\	0	\end{bmatrix}\), we conclude that if \(\begin{bmatrix}	\tvector{y}	\\	t	\end{bmatrix} \in V_*\) then \(\begin{bmatrix}	\tvector{y}	\\	-t	\end{bmatrix} \in V_*\).

%From the assumption that \(V\) is an open ball centered at \(\begin{bmatrix}	\tvector{d}	\\	0	\end{bmatrix}\), we conclude that if \(\begin{bmatrix}	\tvector{y}	\\	t	\end{bmatrix} \in V_*\) then \(\begin{bmatrix}	\tvector{y}	\\	-t	\end{bmatrix} \in V_*\).
%
For any $\begin{bmatrix}	\tvector{y}	\\	t	\end{bmatrix} \in V_*$, if $E(\tvector{p} + t \tvector{y})$ exists then its gradient exists, and the gradient of $F$ is characterized as
\begin{equation*}	%\label{eqn:gradient_F}
	\nabla F(\tvector{y},t) = \begin{bmatrix}t \nabla E(\tvector{p} + t \tvector{y}) \\ \langle \tvector{y}, \nabla E(\tvector{p} + t \tvector{y})\rangle \end{bmatrix}\,,
\end{equation*}
where $\nabla E(\tvector{p} + t \tvector{y})$ is the gradient of $E$ at the point $\tvector{p} + t \tvector{y}$. Decompose $\nabla F(\tvector{y}, t)$ as $\nabla F(\tvector{y}, t) = \tvector{a} + \tvector{b}$, where
\begin{align*}
	\tvector{a}	=	\begin{bmatrix}	\tvector{0}	\\	\langle \tvector{y}, \nabla E(\tvector{p} + t \tvector{y})\rangle	\end{bmatrix}\,,
	&&	\quad\text{and}\quad	&&
	\tvector{b}	=	t \begin{bmatrix}	\nabla E(\tvector{p} + t \tvector{y})	\\	0	\end{bmatrix} \,.
\end{align*} 
By the Cauchy-Schwartz inequality and definition of $\| \cdot \|$,
\begin{align*}
	\|\tvector{a}\|	\le	\|\tvector{y}\| \|\nabla E(\tvector{p} + t \tvector{y})\|,
	&&	\quad \text{and} \quad	&&
	\|\tvector{b}\|	=	|t| \|\nabla E(\tvector{p} + t \tvector{y})\| \,.
\end{align*}
Using the triangle inequality, relate $\|\nabla F(\tvector{y}, t)\|$ and $\|\nabla E(\tvector{p} + t \tvector{y})\|$ by
\begin{align*}
	\|\nabla F(\tvector{y}, t)\|
		\le&	\|\tvector{a}\| + \|\tvector{b}\|
\\
		\le&	\|\tvector{y}\| \|\nabla E(\tvector{p} + t \tvector{y})\| + |t| \|\nabla E(\tvector{p} + t \tvector{y})\|
\\
		=&	(\|\tvector{y}\| + |t|) \|\nabla E(\tvector{p} + t \tvector{y})\|
	\,.
\end{align*}
Note that by definition of $V_*$, $|t| < \frac{1}{2}$ and $\|\tvector{y}\| < \|\tvector{d}\| + \frac{1}{2} = \frac32$, so
\begin{equation*}
	\|\nabla F(\tvector{y}, t)\| \le 2 \|\nabla E(\tvector{p} + t \tvector{y})\|
	\,.
\end{equation*}
%{\color{blue}Explained all "magic constants", simplified proof by removing orthogonality requirement.}
%TODO
We now establish \eqref{eqn:Lojasiewicz_inequality_cone} for points $\tvector{p} + t \tvector{y}$ using \eqref{eqn:linearized_Lojasiewicz}:
\begin{align}
	\left|E(\tvector{p} + t \tvector{y}) - \lim_{s \to 0} E(\tvector{p} + s \tvector{d})\right|^{1 - \theta}
		=&	|F(\tvector{y}, t) - F(\tvector{d}, 0)|^{1 - \theta} \notag
\\
		\le&	\frac{k}{2} \|\nabla F(\tvector{y}, t)\| \notag
\\
		\le&	k \|\nabla E(\tvector{p} + t \tvector{y})\|
	\,.
	\label{eqn:pre_lojasiewicz}
\end{align}
Because this holds, with identical constants, for all $\begin{bmatrix}\tvector{y} \\ t \end{bmatrix} \in V_*$, we may establish \eqref{eqn:Lojasiewicz_inequality_cone} on a subset of $\mathbb{R}^m$. Define the cone $U_{\bowtie}$ by
\begin{align*}
	U_{\bowtie}	&=	\left\{\ \tvector{p} + t \tvector{y} \left| \begin{bmatrix}\tvector{y} \\ t \end{bmatrix} \in V_* \right\}\right. \subset \mathbb{R}^m
	\,.
\end{align*}
Note that $U_{\bowtie}$ is open and satisfies the requirements of the theorem. By \eqref{eqn:pre_lojasiewicz}, for any $\tvector{x} \in U_{\bowtie}$,
\begin{equation*}
	\left|E(\tvector{x}) - \lim_{s \to 0} E(\tvector{p} + s \tvector{d})\right|^{1 - \theta}
		\le	k \|\nabla E(\tvector{x})\|
	\,.
\end{equation*}
This completes the proof.
\end{proof}

Though we have shown that a generalized {\L}ojasiewicz gradient inequality holds in cones near the discontinuities of bounded rational functions, it is important to note that these cones are not neighborhoods of the discontinuities.
Before this theorem may be used to show convergence, it must be shown that a sequence approaching a discontinuity remains within a single cone, or within a finite union of cones.

\subsection{A {\L}ojasiewicz inequality for sequences approaching singularities}	\label{sec:loj_seq_tail}
This section proves conditions under which sequences remain within sets such as those produced by \cref{thm:Lojasiewicz_cones}.
This is broken into \cref{thm:union_of_Lojasiewicz,thm:funnel} and~\cref{thm:Lojasiewicz_on_sequence}.
\Cref{thm:funnel} shows that a continuous function on a product space may be used as a funnel, guiding a sequence into an open subset of one of its factor spaces, such as a set provided by \cref{thm:Lojasiewicz_cones}.
If one set provided by \cref{thm:Lojasiewicz_cones} is insufficient to capture the full behavior of a sequence, \cref{thm:union_of_Lojasiewicz} allows one to form the union of multiple such sets.
Finally, \cref{thm:Lojasiewicz_on_sequence} constructs a set on which the generalized {\L}ojasiewicz gradient inequality holds, and provides conditions under which a sequence will be funneled into the constructed set.

\begin{lemma} \label{thm:union_of_Lojasiewicz}
Let \( U_1, \dots, U_m \) be subsets of \( \mathbb{R}^n \), and let the function \( f : \bigcup_{i=1}^{m} U_i \to \mathbb{R} \) be bounded and differentiable.
If there exist constants \( {L \in \mathbb{R}} \), \({\theta_1, \dots, \theta_m \in (0, \frac12]} \), and \( {c_1, \dots, c_m \in \mathbb{R}} \) such~that, for~all \( {\tvector{x} \in U_i} \),
	\[	|f(\tvector{x}) - L|^{1 - \theta_i}	\le	c_i \| \nabla f(\tvector{x})\|	\,,	\]
then there exist \( \theta_* \in (0, \frac12] \) and \( c_* \in \mathbb{R} \) such that, for~all \( \tvector{x} \in \bigcup_{i=1}^m U_i \),
	\[	|f(\tvector{x}) - L|^{1 - \theta_*}	\le	c_* \| \nabla f(\tvector{x})\|	\,.	\]
\end{lemma}
\begin{proof}
Let \( M = \sup_{\tvector{x} \in \bigcup_{i=1}^m U_i} |f(\tvector{x}) - L| \). Because \( f \) is bounded, \( M < \infty \).
Let \( \theta_* = \min_i \theta_i \). Then, for any \( \tvector{x} \in U_i \),
	\[	|f(\tvector{x}) - L|^{1 - \theta_*}	\quad\le\quad	M^{\theta_i - \theta_*} |f(\tvector{x}) - L|^{1 - \theta_i} \quad\le\quad M^{\theta_i - \theta_*} c_i \| \nabla f(\tvector{x})\|	\,.	\]
Selecting \( c_* = \max_i ( M^{\theta_i - \theta_*} c_i ) \) completes the proof.
\end{proof}

%\newpage
In the proof of \cref{thm:Lojasiewicz_on_sequence}, we will parameterize elements of a sequence \( \{ \tvector{x}_k \}_{k=1}^{\infty} \) as pairs of directions \( \tvector{x}_k / \| \tvector{x}_k \| \) and distances \( \| \tvector{x}_k \| \).
Using this parameterization, we must show that if \( \| \tvector{x}_k \| \) is sufficiently small, then the direction \( \tvector{x}_k / \| \tvector{x}_k \| \) must be within a specified set.
That argument is greatly simplified by \cref{thm:funnel}, which states formally an intuitive property of continuous functions and direct~products of compact sets illustrated by \cref{fig:funnel}.

\begin{figure}[h]
	\centering
	\begin{tikzpicture}[scale=0.6]
	%\draw (4.2, -.25) -- ++(0.1, 0) -- ++(0, 0.25) node[right]{\( \operatorname{d}(x_2, 0) < \epsilon \)} -- ++(0, 0.25) -- ++(-.1, 0);
	
	\filldraw[blue,opacity=0.2,rounded corners=28.8] (-4, -2.5) -- (-4,2) -- (4,2) -- (4,-2.5) -- cycle;
	\draw[blue,rounded corners=28.8, very thick] (-4, -2.5) -- (-4,2) -- (4,2) -- (4,-2.5) -- cycle;
%	\filldraw[blue,opacity=0.3] (-2, -1) -- (-2,1) -- (2,1) -- (2,-1);
	\filldraw[white] (-3, -1) -- (-3,1) -- (3,1) -- (3,-1) -- cycle;
	\filldraw[green!60!black, opacity=0.5] (-3, -1) -- (-3,1) -- (3,1) -- (3,-1) -- cycle;
	\draw[green!60!black, very thick] (-3, -1) -- (-3,1) -- (3,1) -- (3,-1) -- cycle;
	\draw (0,0) node { \( K_1 \times K_2 \) };
	\draw[dashed,very  thick] (-3.5, -1.5) -- (-3.5,1.5) -- (3.5,1.5) -- (3.5,-1.5) -- cycle;
	\draw[dotted,thick] (-3,0) -- ++(-0.25,0) node[above] { \( \epsilon \) } -- ++(-0.25,0);
	\draw[dotted,thick] (3.5,0) -- ++(-0.25,0) node[above] { \( \epsilon \) } -- ++(-0.25,0);
	
	\draw[dotted,thick] (0,1.5) -- ++(0,-0.25) node[right] { \( \epsilon \) } -- ++(0,-0.25);
	\draw[dotted,thick] (0,-1) -- ++(0,-0.25) node[right] { \( \epsilon \) } -- ++(0,-0.25);
	
	\draw[blue!50!black] (0,-2) node { \( f(x_1, x_2) \ge \epsilon \) };
	\end{tikzpicture}
	\caption{Illustration of \cref{thm:funnel}}
	\label{fig:funnel}
\end{figure}

\begin{lemma}	\label{thm:funnel}
Given metric spaces \( X_1, X_2 \), compact sets \( K_1 \subset X_1 \) and \( K_2 \subset X_2 \), and a continuous function \( f : X_1 \times X_2 \to \mathbb{R} \), if \( f|_{K_1 \times K_2} > 0 \), then
there~exists \( \epsilon > 0 \) such that, for~each \(x_1 \in X_1\) and \(x_2 \in X_2\), at least one of the statements
\begin{enumerate}
\item \( f(x_1, x_2) < \epsilon \),
\item \( \operatorname{d}(x_1, K_1) < \epsilon \),
\item \( \operatorname{d}(x_2, K_2) < \epsilon \)
\end{enumerate}
is false.
\end{lemma}
\begin{proof}
Note that the product topology on \( X_1 \times X_2 \) is induced by the metric defined by
	\[	\operatorname{d}_{\infty} ((x_1, y_1), (x_2, y_2))	=	\max \{ \operatorname{d} (x_1, x_2), \operatorname{d} (y_1, y_2) \}	\,.	\]
Because \( f \) is continuous, each \( (x,y) \in K_1 \times K_2 \) admits an open neighborhood
	\[	N^{x,y}_{q,r}	=	\left\{\, (z,w) \,|\, \operatorname{d}(x,z) < q\,	\quad\text{and}\quad	\operatorname{d}(y,w) < r \,\right\}	\]
such~that \( f|_{N^{x,y}_{q,r}} > 0 \).
These sets \( N^{x,y}_{q,r} \) form an open cover of \( K_1 \times K_2 \). By Tychonoff's theorem, \( K_1 \times K_2 \) is compact, so the cover admits a finite subcover \( \{ N^{x_i,y_i}_{q_i,r_i} \}_{i = 1}^{n} \). The closures of the sets in this subcover are compact, by the Heine-Borel property, so their union is also compact.

The set \( N^C = (X_1 \times X_2) \setminus ( \bigcup_{i=1}^n N^{x_i,y_i}_{q_i,r_i} ) \) is closed, so \( \epsilon_1 = \frac{1}{2} \operatorname{d}_{\infty}(N^C, K_1 \times K_2) \) is positive.
Define
	\[	K_{\epsilon_1}	=	\{ (x_1, x_2) \in X_1 \times X_2 | \operatorname{d}_{\infty}((x_1, x_2), K_1 \times K_2) \le \epsilon_1 \}	\,,	\]
and note that \( K_{\epsilon_1} \subset \bigcup_{i=1}^n N^{x_i,y_i}_{q_i,r_i} \). As a closed subset of the union of the closures of  \( \{ N^{x_i,y_i}_{q_i,r_i} \}_{i = 1}^{n} \), the set \( K_{\epsilon_1} \) is compact. Thus \( f \) attains a lower bound \( \epsilon_2 > 0 \) on \( K_{\epsilon_1} \). Define \( \epsilon = \min \{ \epsilon_1, \epsilon_2 \} \).

Given a point \( (x_1, x_2) \in X_1 \times X_2 \), the statements \( \operatorname{d}(x_1, K_1) < \epsilon \) and \( \operatorname{d}(x_2, K_2) < \epsilon \) together imply \( \operatorname{d}_{\infty}((x_1, x_2), K_1 \times K_2) < \epsilon_1 \). This in turn implies \( (x_1, x_2) \in K_{\epsilon_1} \), so \( f(x_1, x_2) \ge \epsilon_2 \ge \epsilon \).
\end{proof}

With the lemmas above, we can formalize a statement that certain sequences are funneled into cones on which the {\L}ojasiewicz inequality holds.
Though we have a statement that the {\L}ojasiewicz inequality holds on cones adjacent to essential singularities, infinitely many such cones may be required to cover a deleted neighborhood of a given point.
\Cref{thm:union_of_Lojasiewicz} can be applied only to finitely many cones, and these cones must have an identical limiting value \(L\).
We can ensure a finite union by imposing a condition which induces a compact set covered by cones, such as condition \eqref{eqn:linear_approach} below.

%%%%%%%%%%%%%%%%%%%%%%%%%%%%%%%%%%%%%%%%%%%%%%%%%%%%%%%%%%%%%%%%
\begin{theorem}	\label{thm:Lojasiewicz_on_sequence}
Let \( f \) be a bounded multivariate rational function on \(\mathbb{R}^n\), and let \( \{\tvector{x}_k\}_{k=1}^{\infty} \subset \mathbb{R}^n \) be a sequence such~that \( \{f(\tvector{x}_k)\}_{k=1}^{\infty} \) is monotonic.
Suppose that \( \tvector{x}^* \) is a cluster point of the sequence \( \{\tvector{x}_k \}_{k=1}^{\infty} \) and there~exists a closed set \( V \subset \mathcal{O}_{\tvector{x}^*} \) such~that for~all sufficiently large \( k \),
\begin{equation}	\label{eqn:linear_approach}
	\frac{\tvector{x}_k - \tvector{x}^*}{\|\tvector{x}_k - \tvector{x}^*\|} \in V \,.%	\tag{A4}
\end{equation}
Then there exist an open neighborhood \( U \) of \( \tvector{x}^* \) and constants \( \theta \in (0, 1/2] \) and \( c \in \mathbb{R} \) such that if \( k \) is sufficiently large and if \( \tvector{x}_k \in U \), then
\[
    \left| f(\tvector{x}_k) - \lim_{i \to \infty} f(\tvector{x}_i) \right|^{1-\theta} \le c\|\nabla f(\tvector{x}_k)\|\,.
\]
\end{theorem}

\begin{proof}
If \( f(\tvector{x}^*) \) is defined, then \( f \) is real-analytic in an open neighborhood of \( \tvector{x}^* \), and the conclusion of the theorem follows directly from the {\L}ojasiewicz gradient inequality.

Assume instead that \( f(\tvector{x}^*) \) is undefined. Then \( \tvector{x}_k \ne \tvector{x}^* \)  for~all \( k \in \mathbb{N} \).

The sequence~\( \{f(\tvector{x}_k)\}_{k=1}^{\infty} \) is monotonic and bounded, so it admits a limit
\[
	L = \lim_{k \to \infty} f(\tvector{x}_k) \,.
\]
Without loss of generality, we may assume that \( \tvector{x}^* = \tvector{0} \) and that every \( \tvector{u} \in V \) has \( \|\tvector{u}\| = 1\). Define \( F(\tvector{x},t) = \lim_{s \to t} f \left( s \tvector{x} \right) \) and note that it is analytic on \( \mathcal{O}_{\tvector{x}^*} \times \mathbb{R} \).

Let \( K \subset \mathcal{O}_{\tvector{x}^*} \) be defined by
\[
	K = V \cap \left\{ \, \tvector{x} \in \mathbb{R}^n \; \left| \; \liminf_{\tvector{z} \to \tvector{x}} \left| F(\tvector{z},0) - L \right| = 0 \, \right.\right\} \,.
\]
As an intersection of a compact set and a closed set, \( K \) is compact. To show that \( K \) is non-empty, take a subsequence \( \{\tvector{x}_{k_i}\}_{i=1}^{\infty} \) such that \( \tvector{x}_{k_i} \to \tvector{x}^* \), 
and note that \( \left\{\left(\tvector{x}_{k_i} / \|\tvector{x}_{k_i}\|, \|\tvector{x}_{k_i}\|\right)\right\}_{i=1}^{\infty} \) admits a cluster point \( (\tvector{u}^*,0) \), for some \( \tvector{u}^* \in V \). 
Because \( F \) is continuous, \( F(\tvector{u}^*,0) = \lim_{i \to \infty} F(\tvector{x}_{k_i} / \|\tvector{x}_{k_i}\|, \|\tvector{x}_{k_i}\|) = L \), and thus \( \tvector{u}^* \in K \).
Further, we may use continuity of \( F \) to calculate the image \( F[K \times \{ 0 \}] = \{ L \} \).

Using \cref{thm:Lojasiewicz_cones}, create collections
\[ \begin{matrix}
	\mathcal{U}_{\tvector{u}} \underset{\mathrm{open}}{\subset} \mathbb{R}^n\,,
&\hspace{1in}&
	\Theta_{\tvector{u}} \in \left( 0, \frac12 \right]\,,
&\hspace{0.4in}\text{and}\hspace{0.4in}&
	\mathcal{C}_{\tvector{u}} \in \mathbb{R}
\end{matrix} \]
such that, for every \( \tvector{u} \in K \),
\[
	\tvector{y} \in \mathcal{U}_{\tvector{u}} \quad \implies \quad |f(\tvector{y}) - L|^{1-\Theta_{\tvector{u}}} \le \mathcal{C}_{\tvector{u}} \|\nabla f(\tvector{y})\|	\,.
\]
By their definition, each \( \mathcal{U}_{\tvector{u}} \) admits an open subset \( \mathcal{V}_{\tvector{u}} \subset \mathcal{U}_{\tvector{u}} \) and a real number \( \epsilon_{1,\tvector{u}} \in (0, 1] \) such that
\begin{equation}	\label{eqn:tiny_cone}
	0 < |t| \le \epsilon_{1,\tvector{u}}	\quad\wedge\quad	\tvector{y} \in \mathcal{V}_{\tvector{u}}	\quad\implies\quad	t \frac{\tvector{y}}{\left\| \tvector{y} \right\|} \in \mathcal{V}_{\tvector{u}}	\,,
\end{equation}
and such that there~exists \(t > 0\) such~that \(t \tvector{u} \in \mathcal{V}_{\tvector{u}}\).
For each~\( \tvector{u} \in K \), let \( \mathcal{V}'_{\tvector{u}} = \left\{ \left. \frac{2}{\epsilon_{1,\tvector{u}}}\tvector{y} \,\right|\, \tvector{y} \in \mathcal{V}_{\tvector{u}} \right\} \).
Note that \(\|\tvector{u}\| = 1\) by assumption on \(V\), and thus that \(\frac{\epsilon_{1,\tvector{u}}}{2} \in \mathcal{V}_{\tvector{u}}\); it follows that \(\tvector{u} \in \mathcal{V}'_{\tvector{u}}\).
The collection~\( \{ \mathcal{V}'_{\tvector{u}} \}_{\tvector{u} \in K} \) is then an open cover of \( K \), and thus admits a finite subcover \( \{ \mathcal{V}'_{\tvector{u}_i} \}_{i=1}^{m} \), where \( \{ \tvector{u}_i \}_{i=1}^{m} \subset K \).
Let
\[	\begin{aligned}
	\epsilon_1	\quad&=\quad	\min_{i = 1, \dots, m} \epsilon_{1,\tvector{u}_i}	\,,
\\
	\mathcal{U}	\quad&=\quad	\bigcup_{i=1}^{m} \mathcal{U}_{\tvector{u}_i}	\,.
\end{aligned}	\]
By \cref{thm:union_of_Lojasiewicz}, there exist $\theta \in (0, 1/2]$ and $c \in \mathbb{R}$ such that
\[
	\tvector{y} \in \mathcal{U}	\quad\implies\quad	|f(\tvector{y}) - L|^{1-\theta} \le c \|\nabla f(\tvector{y})\| \,.
\]

Note that \( \epsilon_1 K = \{\, \epsilon_1 \tvector{u} \,|\, \tvector{u} \in K \,\} \subset \bigcup_{i=1}^{n} \mathcal{V}_{\tvector{u}_i} \subset \mathcal{U} \), and that the distance between a compact set and a closed set is positive if the sets are disjoint. Thus, we may define
\begin{equation}	\label{eqn:def_delta}
	\delta	\quad=\quad	\frac{\operatorname{d}\left(\epsilon_1 K,\, \mathbb{R}^n \setminus \mathcal{U} \right)}{\epsilon_1} \quad>\quad 0	\,.
\end{equation}
Recalling that \(\tvector{u} \in K \implies \|\tvector{u}\| = 1\), and applying \eqref{eqn:tiny_cone} and \eqref{eqn:def_delta} gives that if \( \| \tvector{x}_k \| < \epsilon_1 \) and \( \operatorname{d}(\tvector{x}_k / \|\tvector{x}_k\|, K) < \delta \), then \( \tvector{x}_k \in \mathcal{U} \).

Let \( V_{\delta} = \{\, \tvector{u} \in V \,|\, \|\tvector{u}\| = 1	\wedge	\operatorname{d}(\tvector{u},K) \ge \delta \} \). By \eqref{eqn:linear_approach},
\[
	\frac{\tvector{x}_k}{\|\tvector{x}_k\|} \notin V_{\delta} \iff \operatorname{d}\left(\frac{\tvector{x}_k}{\|\tvector{x}_k\|}, K\right) < \delta \,.
\]

\( V_{\delta} \) is compact, and \( |F - L| > 0\) on \( V_{\delta} \times \{ 0 \}\).
By \cref{thm:funnel}, there~exists \(\epsilon_2 > 0\) such~that \( \| \tvector{x}_k \| < \epsilon_2 \) and \( \left|F\left(\frac{\tvector{x}_k}{\|\tvector{x}_k\|}, \|\tvector{x}_k\|\right) - L\right| < \epsilon_2 \) together imply \( \frac{\tvector{x}_k}{\|\tvector{x}_k\|} \notin V_{\delta} \).
Note that \( F\left(\frac{\tvector{x}_k}{\|\tvector{x}_k\|}, \|\tvector{x}_k\|\right) = f(\tvector{x}_k) \).

If \( k \) is large enough that \( |f(\tvector{x}_k) - L| < \epsilon_2 \) and if \( \| \tvector{x}_k \| < \min \{ \epsilon_1, \epsilon_2 \} \), then \( \operatorname{d}\left(\frac{\tvector{x}_k}{\|\tvector{x}_k\|}, K\right) < \delta \), and \( \tvector{x}_k \in \mathcal{U} \).
Selecting \( U = \{ \tvector{v} \in \mathbb{R}^n \,|\, \| \tvector{v} - \tvector{x}^* \| < \min \{ \epsilon_1, \epsilon_2 \} \} \) completes the proof.
\end{proof}

With the possible exception of \eqref{eqn:linear_approach}, the conditions of \cref{thm:Lojasiewicz_on_sequence} impose no great burden.
Any sequence produced by a hill-climbing algorithm will necessarily have monotonic \( \{ f(\tvector{x}_k) \}_{k = 1}^{\infty} \).

\begin{proposition}    \label{thm:loj_assumption_4_forms}
The set \( \mathcal{O}_{\tvector{x}^*} \) defined in \cref{def:safe_directions} contains every cluster point of the sequence
\( \left\{ \frac{\tvector{x}_k - \tvector{x}^*}{\|\tvector{x}_k - \tvector{x}^*\|} \right\}_{k=1}^{\infty} \)
if and only~if there exists a closed set \( V \subset \mathcal{O}_{\tvector{x}^*} \) such~that \( \frac{\tvector{x}_k - \tvector{x}^*}{\|\tvector{x}_k - \tvector{x}^*\|} \in V \) for~all sufficiently large \( k \).
\end{proposition}
\begin{proof}
If the set \( V \) exists, then the set \( \{\, \tvector{u} \in V \,|\, \|\tvector{u}\| = 1\,\} \) is compact, so the ``if'' direction holds.

If no such set \( V \) exists, then some subsequence \( \{ \tvector{u}_i \}_{i=1}^{\infty} = \left\{ \frac{\tvector{x}_{k_i} - \tvector{x}^*}{\|\tvector{x}_{k_i} - \tvector{x}^*\|} \right\}_{i=1}^{\infty} \) must approach the compact set \( K = \{\, \tvector{u} \in \mathbb{R}^n \,|\, \|\tvector{u}\| = 1\,\} \setminus \mathcal{O}_{\tvector{x}^*} \).
That~is \( \operatorname{d}(\tvector{u}_i, K) \to 0 \).
This may be shown by contradiction; if \( \operatorname{d}(\tvector{u}_i, K) \not\to 0 \), a set of points \( \epsilon \)-distant from \( K \) would fulfill the requirements of \( V \), which would contradict the the assumption that no such set exists.
Define a sequence of closest points in the complement of \( \mathcal{O}_{\tvector{x}^*} \) by selecting, for~each \( i \), a point
\[ \tvector{k}_i \in \{\, \tvector{k} \in K \,|\, \operatorname{d}(\tvector{u}_i, \tvector{k}_i) = \operatorname{d}(\tvector{u}_i, K)\,\} \]
The sequence \( \{ \tvector{k}_i \}_{i=1}^{\infty} \) admits a cluster point, and \( \operatorname{d}(\tvector{u}_i, \tvector{k}_i) \to 0 \), so the sequence \( \{ \tvector{u}_i \}_{i=1}^{\infty} \) also admits a cluster point \( \tvector{u}^* \in K \). By definition of \( K \), \( \tvector{u}^* \notin \mathcal{O}_{\tvector{x}^*} \), so the ``only if'' direction holds.
\end{proof}

\subsection{New convergence theorems}	\label{sec:convergence_result}
In the previous sections, we have established that a generalized {\L}ojasiewicz gradient inequality holds on cones and established conditions under which sequences are funneled into those cones.
We now combine this with \Cref{thm:loj_converge_modified,thm:loj_linear_modified} to establish conditions under which sequences converge to essential singularities of bounded multivariate rational functions.

We say that a sequence \( \{ \tvector{x}_k \}_{k=1}^{\infty} \subset \mathbb{R}^n \) satisfies Assumption~\eqref{eqn:assumption_linearapproach} if
\begin{itemize}
	\item the set \( \mathcal{O}_{\tvector{x}^*} \) in \cref{def:safe_directions} contains every cluster point of the sequence
	\begin{equation}	\label{eqn:assumption_linearapproach}
		\left\{ \frac{\tvector{x}_k - \tvector{x}^*}{\|\tvector{x}_k - \tvector{x}^*\|} \right\}_{k \in \mathbb{N}} \,.	\tag{A4}
	\end{equation}
\end{itemize}
%Though this assumption is well-suited to proofs, the equivalent statement provided by \Cref{thm:loj_assumption_4_forms} provides a more intuitive explanation of the assumption, that it holds sequences which approach \( \tvector{x}^* \) in a manner well-approximated by a linear curve.

\begin{theorem}	\label{thm:loj_converge_cones}
Let \( f \) be a bounded multivariate rational function on \(\mathbb{R}^n\), and let \( \{ \tvector{x}_k \}_{k=1}^{\infty} \subset \mathbb{R}^n \) be a sequence of vectors with a cluster point \( \tvector{x}^* \). If Assumptions~\eqref{eqn:loj_assumption_1}, \eqref{eqn:loj_assumption_2}, and \eqref{eqn:assumption_linearapproach} hold on the tail of the sequence, then \( \tvector{x}^* \) is the limit of \( \{ \tvector{x}_k \}_{k=1}^{\infty} \).
\end{theorem}

%{\color{red} awkwardly phrased. You may need to use a list for the assumptions.}{\color{blue} I've tried a new approach.}
\begin{proof}
Assumption~\eqref{eqn:loj_assumption_1} guarantees that the sequence \( \{ f(\tvector{x}_k) \}_{k \in \mathbb{N}} \) is decreasing.
Assumption~\eqref{eqn:assumption_linearapproach} and \Cref{thm:loj_assumption_4_forms} then fulfill the conditions of \Cref{thm:Lojasiewicz_on_sequence}.
This fulfills condition \eqref{eqn:loj_on_sequence} of \Cref{thm:loj_converge_modified}.

The conditions of \Cref{thm:loj_converge_modified} are thus satisfied, so the conclusions of \Cref{thm:loj_converge_cones} hold.
\end{proof}

The proof of \Cref{thm:loj_converge_cones} merely uses \Cref{thm:Lojasiewicz_on_sequence} to show that the conditions of \Cref{thm:loj_converge_modified} hold, so the same argument provides the following specialization of \Cref{thm:loj_linear_modified}.
%{\color{blue} Perhaps ``specialization'' instead of ``analogue''.}
\begin{theorem}	\label{thm:loj_linear_cones}
Under the conditions of \Cref{thm:loj_converge_cones}, if Assumption~\eqref{eqn:loj_assumption_3} holds, then \( \nabla f(\tvector{x}_k) \to 0 \) and the convergence rate may be estimated as
\begin{equation*}
	\|\tvector{x}^* - \tvector{x}_k\|	=
		\left\{\begin{matrix}
			O(q^k)	&	\quad \text{ if }	&	\theta = \frac12	&	\text{ (for some } 0 < q < 1 \text{),}\\
			O(k^{\frac{-\theta}{1 - 2\theta}})	&	\quad \text{ if }	&	0 < \theta < \frac12	&
		\end{matrix}\right.
\end{equation*}
where \( \theta \) is such that \eqref{eqn:loj_on_sequence} holds.
\end{theorem}
\begin{proof}
Follows immediately from the proof of \Cref{thm:loj_converge_cones} and from the conditions of \Cref{thm:loj_linear_modified}.
\end{proof}

These theorems can be used to show convergence of sequences to cluster points not in a function's domain, but both theorems rely on assumption~\eqref{eqn:assumption_linearapproach}, which may be difficult to verify a priori.

\section{Algorithms, examples, and implications}	\label{sec:application}
Uchmajew et~al.\ have shown that \eqref{eqn:loj_assumption_1}, \eqref{eqn:loj_assumption_2}, and \eqref{eqn:loj_assumption_3} hold for the sequences produced by various optimization algorithms \cite{USCHMA:2015,SCH-USC:2015,LI-US-ZH:2015}.
The reader should note, however, that these results may rely~on additional restrictions on the sequence or on the objective function, which may preclude application to sequences approaching singularities of an objective function.

Notably, in the analysis of the Gradient-Related Projection Method with Line-Search (GRPMLS) from \cite{SCH-USC:2015}, the assumption (A0) required by Corollary~2.9 in \cite{SCH-USC:2015} need only hold on the sequence itself, and thus \eqref{eqn:loj_assumption_1}~and~\eqref{eqn:loj_assumption_2} hold even if the sequence diverges or has a cluster point not in the domain of the objective function.
Unfortunately, the additional conditions that Uschmajew used to prove that \eqref{eqn:loj_assumption_3} holds, in particular that \( \nabla f \) must be Lipschitz continuous on a neighborhood of \( \tvector{x}^* \), do not typically hold if \( \tvector{x}^* \) is an essential singularity.

If the gradient-related projection method with line-search is applied to a bounded multivariate rational function then the sequence produced satisfies \eqref{eqn:loj_assumption_1}~and~\eqref{eqn:loj_assumption_2}.
One must still guarantee the existence of cluster points, and show that \eqref{eqn:assumption_linearapproach} holds.

\subsection{Example: \Cref{fig:rational_gradient_descent}}
\Cref{fig:rational_gradient_descent} illustrates a sequence that maximizes the rational function defined by \((x,y) \mapsto \frac{-xy}{(x^2 + y^2)(1 + x^2 + y^2)}\).
We will employ the results of this paper to establish its convergence.
More precisely, we will examine the equivalent problem of minimizing \(f(x,y) = \frac{xy}{(x^2 + y^2)(1 + x^2 + y^2)}\).

The sequence is produced by the method of \cite{SCH-USC:2015}, and thus satisfies \eqref{eqn:loj_assumption_1}~and~\eqref{eqn:loj_assumption_2}.

Any cluster point of the sequence will occur either in the domain of \(f\) or at its singularity --- that is, at \(\tvector{0}\).
To guarantee convergence in the latter case, we establish that \(\mathcal{O}_{\tvector{x}^*} = \mathbb{R}^n \setminus \{\tvector{0}\}\).
As established in \cref{def:safe_directions}, we may examine the Maclaurin series coefficients of the denominator of \((x,y,t) \mapsto f(tx,ty)\).
The zeroth and first Maclaurin series coefficients are identically \(\tvector{0}\).
The second Maclaurin series coefficient is \(f_2(x,y) = \frac{\mathrm{d}^2}{\mathrm{d}t^2}t^2(x^2 + y^2)(1 + t^2(x^2 + y^2))|_{t=0} = 2 (x^2 + y^2)\), which is nonzero on~\(\mathbb{R}^n \setminus \{\tvector{0}\}\).
Thus \(\mathcal{O}_{\tvector{x}^*} = \mathbb{R}^n \setminus \{\tvector{0}\}\), and \eqref{eqn:assumption_linearapproach} holds regardless of the direction from which the sequence approaches.

Noting that the degree of the denominator \((x^2 + y^2)(1 + x^2 + y^2)\) exceeds that of the numerator \(xy\) and that the sequence of images \(f(\tvector{x}_k) = f(x_k, y_k)\) is negative and decreasing ensures that \(\|\tvector{x}_k\|\) remains bounded.
Thus a cluster point exists.
This cluster point must be the limit of the sequence.
If it is in the domain of \(f\), convergence would be guaranteed by the results of \cite{SCH-USC:2015}.
Otherwise, convergence would be guaranteed by \cref{thm:loj_converge_cones}.

In this example, it is also trivial to establish that the sequence converges to \(\tvector{0}\).
If the sequence were to converge to some \(\tvector{x}^*\) in the domain of \(f\), then \cite[Corrollary~2.11]{SCH-USC:2015} and continuity would provide \(\nabla f(\tvector{x}^*) = \tvector{0}\).
Calculation gives \(\nabla f(x,y) \ne \tvector{0}\), so this case does not occur.

Though this trivial example could be solved using alternate methods, the results of this paper provide more interesting results when a function admits multiple singularities.
As an example of this, the function \((x,y) \mapsto \sum_{i=1}^j w_i f(x - u_i, y - v_i)\), for constants \(w_i, u_i, v_i \in \mathbb{R}\) and \(i=1, \dots, j\), admits multiple singularities \(\tvector{x}^*_i = (u_i, v_i)\), each of which has  \(\mathcal{O}_{\tvector{x}^*_i} = \mathbb{R}^n \setminus \{\tvector{0}\}\).
Any sequence produced by GRPMLS would be precluded from cycling among the singularities by the results of this paper.

\subsection{Convergence in direction}	\label{sec:converge_direction}
If a sequence admits a cluster point, one may ask whether this cluster point is the limit of the sequence.
If a sequence does not admit a cluster point, one may instead ask whether there exists a related convergent sequence that approximates a solution of a corresponding problem.
In this \namecref{sec:converge_direction} we examine a class of problems for which such related convergent sequences may be produced.
%In this section we introduce one such sequence, then show how the results of Uschmajew et~al.\ \cite{SCH-USC:2015} may be combined with the results of this paper to show that the produced sequences converge.

If a sequence fails to converge because it is unbounded, one may optimize instead a homogeneous function of degree 0, such as that defined by
\begin{equation}	\label{eqn:homogenized_objective}
	\bar{f}(\tvector{x})	\quad=\quad	\min_{\alpha \in \mathbb{R}} f(\alpha \tvector{x})	\,.
\end{equation}
A sequence that optimizes \(\bar{f}\) may be paired with a sequence of scalar constants to describe a sequence that optimizes the original objective function \(f\).
In many cases, this is significantly more difficult to analyze than the original objective function \(f\), in part because \(\bar{f}\) may introduce singularities.
For some notable problems, such as low-rank approximation of tensors, \(\bar{f}\) is a bounded multivariate rational function, and is thus within the scope of this paper.
%{\color{blue}TODO Reference (currently moved) proof that it's rational. (maybe have it in an appendix...)}

Optimizing \(\bar{f}\) may still produce a divergent sequence \( \{ \tvector{x}_i \}_{i \in \mathbb{N}} \).
To address~this, consider instead the normalized sequence \( \{ \tvector{x}_i / \|\tvector{x}_i\| \}_{i \in \mathbb{N}} \), which also optimizes \(\bar{f}\).
This normalized sequence must admit a cluster point, by compactness of the unit ball in a finite-dimensional Banach space.
We will show that if Assumption~\eqref{eqn:loj_assumption_1} holds on \( \{ \tvector{x}_i \}_{i \in \mathbb{N}} \), it holds also on the normalized sequence \( \{ \tvector{x}_i / \|\tvector{x}_i\| \}_{i \in \mathbb{N}} \).
%This normalized sequence must admit a cluster point, by compactness of the unit ball in a finite-dimensional Banach space.

\begin{lemma}	\label{thm:normalization_bound}
Let \(V\) be a Hilbert space, and let \( \tvector{u}, \tvector{v} \in V \) be vectors such that \( \|\tvector{u}\| = 1\).
Then
\begin{equation}	\label{eqn:normalization_bound}
	\left\|\tvector{u} - \frac{\tvector{v}}{\|\tvector{v}\|}\right\|	\quad\le\quad	2 \|\tvector{u} - \tvector{v}\|	\,.
\end{equation}
\end{lemma}
\begin{proof}
Let \(k = \|\tvector{v}\|\), let \( \tvector{w} = \tvector{v} / k \), and let \(x = \operatorname{Re}(\langle \tvector{u}, \tvector{w}\rangle)\).
\begin{align}
		\eqref{eqn:normalization_bound}	& &	\iff	\notag
\\
	\textstyle\frac{1}{2} ( 4 \left\|\tvector{u} - k \tvector{w}\right\|^2 - \left\|\tvector{u} - \tvector{w}\right\|^2 ) &	\ge 0	&	\iff	\notag
\\
	2 k^2 - 4 k x + 1 + x &	\ge 0	&	\iff	\notag
\\
	2 (k - x)^2 + 1 + x - 2x^2 &	\ge 0	&	\label{eqn:diff_of_norms}
\end{align}
If \( \frac{-1}{2} \le x \le 1 \), then \( 1 + x - 2 x^2 \ge 0 \), and \eqref{eqn:diff_of_norms} holds.
%If \(\operatorname{Re}(\langle \tvector{u}, \tvector{w}\rangle) \ge 0\), the Cauchy-Schwarz inequality implies that \(\eqref{eqn:diff_of_norms} \ge 0\).
\\
If \( x \le 0 \), then \( (k - x)^2 \ge x^2 \) by the choice of~\(k\), and \eqref{eqn:diff_of_norms}~holds if~\( 1 + x \ge 0 \).
\\
By the Cauchy-Schwarz inequality, \(x \in [-1,1]\), so \eqref{eqn:diff_of_norms} holds.
\end{proof}

\begin{proposition}
Suppose that \( f : \mathbb{R}^n \to \mathbb{R} \) has the property that \( f(\tvector{x}) = f(c\tvector{x}) \) for~all \( c \in \mathbb{R} \). Suppose also that \( \{ \tvector{x}_i \}_{i \in \mathbb{N}} \) is a sequence such that Assumption~\eqref{eqn:loj_assumption_1} holds. Then Assumption~\eqref{eqn:loj_assumption_1} holds also for the sequence \( \{ \tvector{x}_i / \| \tvector{x}_i \| \}_{i \in \mathbb{N}} \).
\end{proposition}
\begin{proof}
For clarity, let \( \tvector{u}_i = \tvector{x}_i / \| \tvector{x}_i \| \). By Assumption~\eqref{eqn:loj_assumption_1}, there is some \( \sigma > 0 \) such that
	\[ f(\tvector{x}_i) - f(\tvector{x}_{i+1}) \ge \sigma \| \nabla f(\tvector{x}_{i})\| \, \|\tvector{x}_{i} - \tvector{x}_{i+1} \| \,. \]
It is easily verified, by definitions of the derivative and \(f\), that
	\[ \nabla f(\tvector{u}_i) = \| \tvector{x}_i \| \nabla f(\tvector{x}_i) \, . \]
Further, by \cref{thm:normalization_bound},
	\[ \| \tvector{u}_{i} - \tvector{u}_{i + 1} \| \quad\le\quad 2 \left\| \tvector{u}_{i} - \frac{\tvector{x}_{i+1}}{\|\tvector{x}_{i}\|} \right\| \quad=\quad 2 \frac{\| \tvector{x}_{i} - \tvector{x}_{i+1} \|}{\| \tvector{x}_i \|} \,.\]
Thus,
\[\begin{aligned}
	&	\frac{\sigma}{2} \| \nabla f(\tvector{u}_{i})\| \, \|\tvector{u}_{i} - \tvector{u}_{i+1} \|	\\
	=\quad&	\frac{\sigma}{2} \| \nabla f(\tvector{x}_{i})\| \, \|\tvector{x}_{i}\| \, \|\tvector{u}_{i} - \tvector{u}_{i+1} \|	\\
	\le\quad&	\sigma \| \nabla f(\tvector{x}_{i})\| \, \|\tvector{x}_{i} - \tvector{x}_{i+1} \|	\\
	\le\quad&	f(\tvector{x}_i) - f(\tvector{x}_{i+1})	\\
	=\quad&	f(\tvector{u}_i) - f(\tvector{u}_{i+1})	\,,
\end{aligned} \]
So Assumption~\eqref{eqn:loj_assumption_1} holds for the sequence \( \{ \tvector{u}_i \}_{i \in \mathbb{N}} \).
\end{proof}

\subsection{Implications for tensor approximation}
Creation of low-rank approximations of tensors, specifically when approximating a tensor using two or more separable tensors, is prone to produce divergent sequences of parameters.
We make no effort to establish convergence of such a sequence.
Instead, we consider the convergence of the sequence of normalized parameters, as in \cref{sec:converge_direction}.

Let \(\tau(\tvector{x})\) be a multilinear map from parameters \(\tvector{x} \in \mathbb{R}^n\) to tensors, and let \(\ttensor{T}\) be the tensor one wishes to approximate.
If one selects the usual objective function \(f(\tvector{x}) = \|\tau(\tvector{x}) - \ttensor{T}\|^2\), one obtains a polynomial function on \(\mathbb{R}^n\).
If instead one selects \(\hat{f}(\tvector{x}) = \left\|\frac{\langle\tau(\tvector{x}), \ttensor{T}\rangle}{\|\tau(\tvector{x})\|^2} \tau(\tvector{x}) - \ttensor{T}\right\|^2\), one obtains a bounded multivariate rational function with denominator \(\|\tau(\tvector{x})\|^4\).
Moreover, \(\hat{f}\) is a homogeneous function of degree 0, and \(\hat{f}(\tvector{x}) = \min_{\alpha \in \mathbb{R}} f(\alpha \tvector{x})\) wherever \(\hat{f}\) is defined.

A full analysis of the Maclaurin series coefficients \(\frac{\mathrm{d}^m}{\mathrm{d}t^m} \|\tau(\tvector{x + t \tvector{d}})\|^4\) would be beyond the scope of this paper.
We can, however, draw conclusions from established properties of the set \(\mathcal{O}_{\tvector{x}^*}\).
In particular, \(\mathbb{R}^n \setminus \mathcal{O}_{\tvector{x}^*}\) is of Lebesgue measure 0, so it covers almost every direction of approach.

If one uses the GRPMLS method from \cite{SCH-USC:2015} to optimize \(\hat{f}\), and normalizes the resulting parameters, then it is likely that the resulting sequence of parameters will converge.
That is, it is likely that the approximation's summands will form a stable configuration, changing little except in scale.

\section{Concluding remarks}
We have shown that a generalized {\L}ojasiewicz inequality holds on sets adjacent to essential singularities of bounded multivariate rational functions.
%We have further shown that under the technical condition \eqref{eqn:assumption_linearapproach} and a monotonicity condition, sequences will remain within sets on which the the generalized {\L}ojasiewicz inequality holds.
%We have shown that a generalized {\L}ojasiewicz inequality holds on the tail of a sequence, provided that the objective function is bounded and rational and that Assumption~\eqref{eqn:assumption_linearapproach} holds.
With this result, we have provided sufficient conditions under which a sequence will converge to a singularity of the objective function.
Finally, we have shown that these results may be employed in the study of unbounded sequences by converting to a projective space.

The results employed in this paper could potentially be extended to sequences for which Assumption~\eqref{eqn:assumption_linearapproach} does not hold.
In particular, we expect that if a sequence is well-described by some smooth curve, that is \({\tvector{x}_i = \gamma(1 / i)}\) for some smooth \(\gamma\) with \(\gamma(0) = \tvector{x}^*\), then a generalized {\L}ojasiewicz inequality should hold on that sequence.
This would, however, be of little practical utility due to the difficulty of showing that an algorithm produces sequences well-described by smooth curves.

\section*{Acknowledgement}
This paper was completed with the guidance of Martin J.\ Mohlenkamp.

\bibliographystyle{siamplain}
\raggedright
\bibliography{nrefs,sor,lojasiewicz}

\appendix
\section{Full proof of \cref{thm:loj_converge_modified,thm:loj_linear_modified}}	\label{apx:full_proof_lojasiewicz}
\begin{proof}
Assumption~\eqref{eqn:loj_on_sequence} implies that \( \lim_{k \to \infty} f(\tvector{x}_k) \) exists.
The gradient of \( f \) is finite where defined, so assumption \eqref{eqn:loj_on_sequence} requires that \( \displaystyle\lim_{k \to \infty} f(\tvector{x}_k) \) be finite.

A proof of the original theorems of Absil, Uschmajew, et~al.\ is provided in \cite[p.~644]{SCH-USC:2015}. 
Because only three changes must be made to the original proof, I quote Uschmajew's proof here as indented block quotes.\footnote{To simplify reading, I have changed Uschmajew's notation to match mine.}
Between block quotes, I supply the alterations which extend Uschmajew's proof to a proof of \Cref{thm:loj_converge_modified,thm:loj_linear_modified}.
Let \(f_k = f(\tvector{x}_k)\), and let \(g^-_k = \|\nabla f(\tvector{x}_k)\|\).
\begin{quotation}
We can assume that \(g^-_k > 0\) for~all \(k\) since otherwise the sequence becomes stationary and there is nothing to prove.
\end{quotation}
There will also be no loss of generality to assume that \eqref{eqn:loj_assumption_1} and \eqref{eqn:loj_assumption_2} hold for~all \(k\) and that \(\lim_{k \to \infty} f_k = 0\).
Then \(0 \le f_k\) for~all \(k\) and the {\L}ojasiewicz gradient inequality at \(\tvector{x}^*\) reads as
\begin{equation}	\label{eqn:Adot1}
	f_k^{1 - \theta}	\le	\Lambda g^-_k	\tag{A.1}
\end{equation}
whenever \(\tvector{x}_k \in U\). The set \(U\) contains an open ball \(B_{\delta}(\tvector{x}^*)\) for some \(\delta > 0\).
\begin{quotation}
Let \(\epsilon \in (0, \delta]\), and assume \(\|\tvector{x}_k - \tvector{x}^*\| < \delta\).
Then, by \eqref{eqn:Adot1} and \eqref{eqn:loj_assumption_1},
	\[	\|\tvector{x}_k - \tvector{x}_{k+1}\|	\le	\frac{\Lambda}{\sigma} f_k^{\theta - 1}(f_k - f_{k+1})	\,.	\]
Using the fact that for \({\phi \in [f_{k+1},f_k]}\) there holds \({f_k^{\theta-1} \le \phi^{\theta - 1} \le f_{k+1}^{\theta-1}}\), we can estimate
	\[	f_k^{\theta - 1}(f_k - f_{k+1})	\le	\int_{f_{k+1}}^{f_k} \phi^{\theta-1} \operatorname{d}\phi	=	\frac{1}{\theta}(f_k^{\theta} - f_{k+1}^{\theta})	\]
and thus obtain
	\[	\|\tvector{x}_k - \tvector{x}_{k+1}\|	\le	\frac{\Lambda}{\sigma \theta} (f_k^{\theta} - f_{k+1}^{\theta})	\,.	\]
More generally, let \(\|\tvector{x}_j - \tvector{x}^*\| < \epsilon \le \delta\) for \(k \le j < m\); we get by this argument that
\begin{equation}	\label{eqn:Adot2}
	\|\tvector{x}_m - \tvector{x}_k\|
		\le	\sum_{j = k}^{m} \|\tvector{x}_{j+1} - \tvector{x}_{j}\|
		\le	\sum_{j = k}^{m} \frac{\Lambda}{\sigma \theta} (f_j^{\theta} - f_{j+1}^{\theta})
%		=	\frac{\Lambda}{\sigma \theta} (f_k^{\theta} - f_{m}^{\theta})
		\le	\frac{\Lambda}{\sigma \theta} f_k^{\theta}
	\,.	\tag{A.2}
\end{equation}
\end{quotation}
Since \(\tvector{x}^*\) is an accumulation point, and because \(f_k \to 0\), we can pick \(n\) so large that
\begin{quotation}
	\[	\|\tvector{x}_k - \tvector{x}^*\| < \frac{\epsilon}{2}	\quad\text{and}\quad	\frac{\Lambda}{\sigma \theta} f_k^{\theta} < \frac{\epsilon}{2}	\,.		\]
Then \eqref{eqn:Adot2} inductively implies \(\|\tvector{x}_m - \tvector{x}^*\| < \epsilon\) for all \(m > k\).
This proves that \(\tvector{x}^*\) is the limit point of the sequence, and, by \eqref{eqn:loj_assumption_3}, \(g^-_k \to 0\).

To estimate the convergence rate, let \(r_k = \sum_{j=k}^{\infty} \|\tvector{x}_{j+1} - \tvector{x}_j\|\).
Then \(\|\tvector{x}_{k} - \tvector{x}^*\| \le r_k\), so it suffices to estimate the latter. By \eqref{eqn:Adot2}, \eqref{eqn:Adot1}, and \eqref{eqn:loj_assumption_3}, there exists \(k_0 \ge 1\) such~that for \(k \ge k_0\) it holds that
	\[	r_n^{\frac{1-\theta}{\theta}}	\le	\left(\frac{\Lambda}{\sigma \theta}\right)^{\frac{1-\theta}{\theta}} \frac{\Lambda}{\kappa} \|\tvector{x}_{k+1} - \tvector{x}_{k}\|	=	\left(\frac{\Lambda}{\sigma \theta}\right)^{\frac{1-\theta}{\theta}} \frac{\Lambda}{\kappa} (r_k - r_{k+1})	\,,	\]
that is,
\begin{equation}	\label{eqn:Adot3}
	r_{k+1}	\le	r_k - \nu r_{k}^{\frac{1-\theta}{\theta}}	\tag{A.3}
\end{equation}
with \(\nu = \left(\frac{\Lambda}{\sigma \theta}\right)^{\frac{\theta-1}{\theta}} \frac{\kappa}{\Lambda}\).
Now, if \(\theta = 1/2\), we get from \eqref{eqn:Adot3} that \(\nu \in (0,1)\), and
	\[	r_k	\le	r_{k_0} (1 - \nu)^{k - k_0}	\]
for \(k \ge k_0\).
The case \(0 < \theta < 1/2\) is more delicate.
We follow Levitt: put \(p = \frac{\theta}{1-2\theta}, C \ge \max\{(\frac{\nu}{p})^{-p}, r_{k_0} k_0^{-p}\}\), and \(s_k = C k^{-p}\); then \(s_{k_0} \ge r_{k_0}\), and
\begin{equation*}
	s_{k+1}
		=	s_{k}(1 + k^{-1})^{-p}
		\ge	s_k (1 - p k^{-1})
		=	s_k - \frac{p}{C^{1/p}} s_k^{\frac{p + 1}{p}}
		\ge	%s_k - \nu s_k^{\frac{p+1}{p}}	=
			s_k - \nu s_k^{\frac{1 - \theta}{\theta}}
\end{equation*}
(the first inequality holding by convexity of \(x^{-p}\)).
Using induction, it now follows from \eqref{eqn:Adot3} that \(r_k \le s_k\) for~all \(k \ge k_0\), which finishes the proof.
\end{quotation}
\end{proof}

\end{document}